\documentclass[9pt,a4paper,reqno]{amsart}
\usepackage[T1]{fontenc}
\usepackage{hyperref}
\usepackage{amssymb}
\usepackage{xcolor}
\usepackage{amsmath}
\usepackage{amsfonts}
\usepackage{ifthen}
\usepackage{leftidx}
\usepackage{comment}
\usepackage{mathtools}
\usepackage{tensor}
\usepackage{todonotes}
\newtheorem{innercustomthm}{Theorem}
\newenvironment{customthm}[1]
  {\renewcommand\theinnercustomthm{#1}\innercustomthm}
  {\endinnercustomthm}
\usepackage[mathscr]{euscript}

\newcounter{margnotes}

\newcommand{\Mitya}[1]{\todo[size=\tiny,inline,color=cyan]{#1 \\ \hfill --- Mitya}}

\def\sideremark#1{\ifvmode\leavevmode\fi\vadjust{\vbox to0pt{\vss 
			\hbox to 0pt{\hskip\hsize\hskip1em           
				\vbox{\hsize3cm\tiny\raggedright\pretolerance10000
					\noindent #1\hfill}\hss}\vbox to8pt{\vfil}\vss}}}%

\newcounter{lemenumi}
\newcommand{\labelemenumi}{(\alph{lemenumi})}

\newtheorem{theorem}{Theorem}[section]
\newtheorem{lemma}[theorem]{Lemma}

\newtheorem{proposition}[theorem]{Proposition}
\newtheorem{corollary}[theorem]{Corollary}

\theoremstyle{definition}

\newtheorem{definition}[theorem]{Definition}
\newtheorem{problem}[theorem]{Problem}
\newtheorem{example}[theorem]{Example}

\theoremstyle{remark}
\newtheorem{remark}[theorem]{Remark}

\numberwithin{equation}{section}

\newcommand{\A}{\mathcal{A}}

\newcommand{\CC}{\mathcal{C}}

\def\de#1{\operatorname{\tau}\left(#1\right)}
\def\val#1{\operatorname{\lambda}\left(#1\right)}
\def\ev{\operatorname{ev}}
\def\ev{\operatorname{ev}}

\newcommand{\p}{\mathbb{P}} 
\newcommand{\scp}{\mathscr{P}} 
\newcommand{\pe}{P_\eta}

\newcommand{\id}{\mathrm{id}}

\newcommand{\sgn}{\mathrm{sgn}}
\newcommand{\OO}{\mathcal{O}}

\begin{document}
	\newcommand{\xc}{\xi}
	\newcommand{\yc}{\eta}
	\newcommand{\ep}{\epsilon}
	\newcommand{\jb}{II}
	\newcommand{\df}{d}
	\newcommand{\al}{\alpha}
	
	\newcommand{\testrat}{R}
	\newcommand{\testpol}{P}
	
	\newcommand{\pn}{Q}
	\newcommand{\qn}{Q_1}
	\newcommand{\ve}{\epsilon}
	\newcommand{\varx}{\mathrm{var}}
	\newcommand{\org}{T}
	\newcommand{\orgg}{t}
	\newcommand{\supp}{\text{supp}}
	\newcommand{\R}{\mathbb{R}}
	\newcommand{\C}{\mathbb{C}}
	\newcommand{\I}{\mathbb{I}}
	\newcommand{\N}{\mathbb{N}}
	\newcommand{\Z}{\mathbb{Z}}
	\newcommand{\Q}{\mathbb{Q}}
        \newcommand{\T}{\mathbb{T}}        
	\newcommand{\Xbar}{\bold{X}}
     \newcommand{\Hol}{\mathcal{H}}

	
	\newcommand{\edg}{\gamma^0}
	\newcommand{\vtx}{p}
	\newcommand{\CH}{{\mathbb{C}} H_1({\mathcal{O}})}
	\newcommand{\CHy}{{\mathbb{C}} H_1({\mathcal{O}_y})}
	\title[Melnikov functions]{Noetherianity and Length of Melnikov functions}
 \today

	\author[P. Marde\v si\'c]{Pavao Marde\v si\'c}
\address[a]{Universit\'e Bourgogne Europe, Institute de 
	Math\'ematiques de Bourgogne - UMR 5584 CNRS\\
	Universit\'e Bourgogne Europe,
	9, avenue Alain Savary,
	BP 47870, 21078 Dijon \\FRANCE}
    \address[a]{Laboratorio Internacional Solomon Lefschetz, IRL 2001, CNRS-UNAM, Instituto de Matematicas, Unidad Cuernavaca, México. }
\address[a]{ University of Zagreb, Faculty of Science, Department of Mathematics, Bijeni\v cka 30, 10 000 Zagreb, Croatia}
\email{pavao.mardesic@ube.fr}

\author[D. Novikov]{Dmitry Novikov}
\address[b]{Faculty of Mathematics and 
	Computer Science, Weizmann Institute of Science, Rehovot, 7610001\\Israel}
\address[b]{Institute for Advanced Study,
Princeton, NJ
08540 USA
}	\email{dmitry.novikov@weizmann.ac.il}

\author[L. Ortiz-Bobadilla]{Laura Ortiz-Bobadilla}
\address[c]{Instituto de Matem\'aticas, Universidad Nacional Aut\'onoma de 
M\'exico 	(UNAM), 	\'Area de la Investigaci\'on Cient\'ifica, Circuito 
exterior, Ciudad 	Universitaria, 04510, Ciudad de M\'exico, M\'exico}
\email{laura@im.unam.mx}

\author[J. Pontigo-Herrera]{Jessie Pontigo-Herrera}
\address[d]{Instituto de Matem\'aticas, Universidad Nacional Aut\'onoma de 
	M\'exico 	(UNAM), 	\'Area de la Investigaci\'on Cient\'ifica, Circuito 
	exterior, Ciudad 	Universitaria, 04510, Ciudad de M\'exico, M\'exico}
\email{pontigo@matem.unam.mx}
	
	
\thanks{
The first author benefited from the support of the IRL 2001 Solomon Lefschetz CNRS-UNAM,  Croatian Science Foundation (HRZZ) grant IP-2022-10-9820, partial support by the Horizon grant 101183111-DSYREKI-HORIZON-MSCA-2023-SE-01,
This work was supported by Kovner Member Fund and IAS, by ISF grant 1167/17, Minerva grant 714141 and
 Papiit (Dgapa UNAM) IN103123}	

\begin{abstract}
    We study foliations in $\C^2$ given by polynomial deformations of the form $dH+\epsilon \eta=0$, with $\gamma(t)\subset H^{-1}(t)$ a family of cycles. The \emph{Poincaré first return map} is of the form $P(t)=t+\sum_j \epsilon^j M_j^\gamma(t).$ The functions $M_j^\gamma$ are called \emph{Melnikov functions} and are given by \emph{iterated integrals of orbit length} at most $j$. 

    We show that, for each $k\in\N$, there exists a \emph{universal Noetherianity index} $n_{\scriptscriptstyle H,\gamma}(k)$, independent of the deformation $\eta$, such that, if $M_j^\gamma\equiv0$, for $j=1,\ldots,n_{ H,\gamma}(k)$, then $M_j^\gamma$ is of orbit length $j-k$, for any Melnikov function $M_j^\gamma$.
    We call the smallest index with this property just the \emph{Noetherianity index} $\nu_{\scriptscriptstyle H,\gamma}(k)$.

    In order to prove this theorem, we develop a structure theorem for Melnikov functions and use the Ritt-Raudenbush differential algebra theorem.

    We calculate the universal Noetherianity index $n_{H,\gamma}(k)$ in various nontrivial examples. 
\end{abstract}

\maketitle
\providecommand{\keywords}[1]
{
  \small	
  \textbf{\textit{Keywords---}} #1
} 

\section{Introduction}
Let $H$ be a polynomial, $H\in\R[x,y]$, and let  $\mathcal{F}_\epsilon$ be a foliation in $\mathbb{R}^2$ defined by a deformation 
\begin{equation*}
    dH+\epsilon\eta=0, 
\end{equation*}
where $\eta$ is a polynomial one-form and $\epsilon$ is a small parameter. Assume that the foliation $dH=0$ in $\mathbb{R}^2$ has a continuous family of cycles $\gamma(t)\subset \{H=t\}$. The infinitesimal version of the Poincaré center-focus problem and of Hilbert's 16th problem deal with the determination of the behavior of the family of cycles $\gamma(t)$ by the deformation $\mathcal{F}_\epsilon$. 
Will the curves of the family $\gamma(t)$ after perturbation remain as cycles? If not, how many limit cycles (i.e., isolated cycles) are born?

To study these problems, one considers the \emph{Poincaré first return map} 
\begin{equation}\label{eq.First-return}
    P_\epsilon^{\gamma}(t)=t+\epsilon M^\gamma_1(t)+\epsilon^2 M^\gamma_2(t)+\cdots
\end{equation}
of the foliation $\mathcal{F}_\epsilon$, with respect to $\gamma$.

Each fixed point of $P^\gamma_\epsilon$, corresponds to a cycle.
Now we complexify the problem. That is, we consider $H\in\C[x,y]$ and $\gamma(t)$ a loop in a regular fiber $\{H=t\}$, and $\eta$ a complex polynomial one-form. In the complex setting, there can also exist periodic, non-fixed points of $P^\gamma_\epsilon$, which also correspond to cycles of the foliation \cite{I69}.

The functions $M_j$ in \eqref{eq.First-return} were studied by Poincar\'e \cite{P90} and Pontryagin \cite{Po}, but are commonly called
 \emph{Melnikov funtions}. They are univalued in a neighborhood of a regular value of $H$, but multivalued in $\mathbb{C}\setminus\Sigma_H$, where $\Sigma_H\subseteq\C$ is the finite set of \emph{atypical values} of $H$.

The Melnikov functions contain essential information about what happens with the deformation of the continuous family $\gamma(t)$.  From the Fran\c coise algorithm  \cite{F,FP,G05,M22}  it is known that $M_j$ are \emph{iterated integrals of length} at most $j$. The complexity of these functions is codified by their length as iterated integrals \cite{C}. When the Poincaré map \eqref{eq.First-return}
is not equal to the identity, then the number of zeros of the first non-zero Melnikov function $M_j$ provides an upper bound for the fixed points of the Poincaré map $P_\epsilon(t)$, for $\epsilon$ small enough, in a neighborhood of a regular value.

In \cite{MNOP}, we have shown that 
the length of the first non-zero Melnikov function $M_j$, as an iterated integral is bounded by a constant depending only on the Hamiltonian $H$ and the orbit under monodromy of $\gamma$. That constant $\kappa$ was called \emph{the orbit depth}, and it was defined in terms of the position of the orbit of $\gamma$ with respect to the lower central sequence of the first homotopy group of the regular fiber of $H$. 
However, the orbit depth can also be infinity, as shown in \cite{MNOP2}.

\medskip
In this article, we deal with the problem of characterizing the length of all of the Melnikov functions in \eqref{eq.First-return}, not only the first non-zero.
We also give conditions under which this length can be reduced.

We first prove a structure Theorem \ref{Theo.M-sigma_k}, showing that any Melnikov function on any cycle can be expressed through Melnikov functions along some basic cycles and their derivatives.

The key discovery  in this paper is a new \emph{Noetherianity phenomenon} present in the set of Melnikov functions. More precisely, we introduce the notion of \emph{length with respect to the orbit} (shorter \emph{orbit length}) of Melnikov functions. We show that Melnikov functions can be expressed as evaluations of some differential operators depending on $H$ and $\gamma$ only. These operators act on a tuple of holomorphic functions dependent on $\eta$. By Theorem~\ref{Theo.M-sigma_k} and the Ritt-Raundenbush Theorem \ref{thm:Ritt-Raudenbush}, the algebra of such differential operators has some strong  Noetherian properties, which imply the drop of the orbit length of Melnikov functions. More precisely, we show in Theorem \ref{thm:main1} that, for any $k\in\mathbb{N}$, there exists a \emph{universal Noetherianity index} $n=n_{H,\gamma}(k)\in\mathbb{N}$, depending only on $H,\gamma$ and $k$ and independent of $\eta$, such that the identical vanishing of the first $n$ Melnikov functions implies that the orbit length of $M_j$ is at most $j-k$, for every $j\in\N$. We call the smallest such number $n$, the \emph{Noetherianity index} $\nu_{H,\gamma}(k).$

We provide examples in which we calculate the universal Noetherianity index $n_{H,\gamma}(k)$. One observes in the examples that we studied that this number distinguishes the degeneracy of the undeformed system given by $H$ and $\gamma$. We show that the universal Noetherianity index is nontrivial even in the case when the orbit depth $\kappa$ was infinite.

\section{Main results}

Let $H\in\C[x,y]$ be a  Hamiltonian, and $\gamma(t_0)\subset\{H=t_0\}$ a loop, where $t_0$ is a regular value of $H$. Let $<\delta_i>_{i=1}^m$ be a basis of cycles in the first homology group of the fiber $H^{-1}(t_0)$ and is therefore a set of generators of $L_1=\pi_1\{H=t_0\}$.

We consider a deformation $\mathcal{F}_\epsilon=\{dH+\epsilon\eta=0\}$, where $\eta$ is a polynomial one-form and $\epsilon\in\C$.

 Let $\Gamma=(\C,t_0)$ be a transversal to $H^{-1}(t_0)$ at a point $p_0\in \gamma(t_0)$, parametrized by the values of $H$. Consider the first return map in the foliation $\mathcal{F}_\epsilon$, with respect to the loop $\gamma(t_0)\in\pi_1(H^{-1}(t_0), p_0)$, 
 \begin{equation*}
     P^\epsilon_{\gamma}(t)=t+\epsilon M^\gamma_1(t)+\epsilon^2 M^\gamma_2(t)+\cdots .
 \end{equation*}
 
The Melnikov functions $M^\gamma_j(t)$ admit an explicit expression as follows. 
In a standard way we can identify $\pi_1(H^{-1}(t_0), p_0) $ with $\pi_1(H^{-1}(t), p) $, where $p=\Gamma\cap H^{-1}(t)$, for all $t$ sufficiently close to $t_0$. This defines a continuous family of loops $\gamma(t)\subset H^{-1}(t)$, and the Melnikov functions $M^\gamma_j(t)$ are evaluations of iterated integrals along this family of loops $\gamma(t)$, see \cite{F,FP,G05}.

More generally, the above identification defines the \emph{monodromy action} of the foliation on $\pi_1(H^{-1}(t_0), p_0)$. Take a  loop $\rho\subset\mathbb{C}\setminus\Sigma_H$, $\rho(0)=\rho(1)=t_0$,  where $\Sigma_H\subseteq\C$ is the set of atypical values  of $H$, and let the loop $\tilde\rho(s)\subset\C^2\setminus H^{-1}(\Sigma_H)$, $\tilde\rho(0)=\tilde\rho(1)=p_0$, be any lift of $\rho$, $\rho=H\circ\tilde\rho$. The (homotopy class of the) pair $(\rho, \tilde\rho)$ defines an automorphism $Mon_\rho: \pi_1(H^{-1}(t_0), p_0)\to \pi_1(H^{-1}(t_0), p_0)$, with different choices of $\tilde\rho$ defining automorphisms differing by an inner automorphism of $\pi_1(H^{-1}(t_0), p_0)$. The group of all such automorphisms is called the \emph{monodromy group}  of the foliation defined by $H$.

 Denote by $\OO$ the \emph{orbit} of $\gamma(t_0)$ in $\pi_1(H^{-1}(t_0), p_0) $ under the monodromy group of the fibration given by $H$, 
 and let  $L_i$ be the elements of the lower central sequence $L_{i+1}=[L_1, L_i]$, where $L_1=\pi_1(\{H=t_0\},p_0)$.
    Let $\sigma_{i_j}$ be the generators of the space 
\begin{equation}\label{O_i}
        \OO_i:=\frac{\OO\cap L_i}{\OO \cap L_{i+1}}\otimes\C, \quad i\in\N.
\end{equation}

By analytic continuation, the Melnikov functions $M_j$ extend to the whole orbit $\OO$.

\begin{definition}\label{def:length/orbit}
Let $\omega_1,\ldots,\omega_r$, be polynomial 1-forms in $\C^2$ and consider the iterated integral $I=\int\omega_1\cdots\omega_r$. We say that $I$ is an iterated integral \emph{of length $\ell$ with respect to the orbit $\OO$} or, shorter, of \emph{orbit length} $\ell$, if $\ell$ is the smallest integer such that the integral $I^\sigma(t)=\int_{\sigma(t)}\omega_1\cdots\omega_r$ identically vanishes, for all $i>\ell$ and all  $\sigma\in\OO_{i}$.
\end{definition}

If $\OO=\pi_1$ this notion coincides with the classical notion of length of iterated integrals, as iterated integrals of length $j$   are dual to the quotient ${L_{j}/L_{j+1}}$ by Chen's theory \cite{C}. Note that, the orbit length of an iterated integral is at most its length.

The main theorem proved in this article is the following

\begin{customthm}{A}\label{thm:main1}
Let $H\in\C[x,y]$ be a polynomial, and let $\gamma\subset H^{-1}(t_0)$ be a loop. Then, for any $k\in \N$, there exists some $n=n_{H,\gamma}(k)\in \N$ with the following property.

Consider the deformation
   \begin{equation}\label{eq:deformation}
       dH+\epsilon\eta=0,
   \end{equation}
where   $\eta$ is  a polynomial one-form in $\C^2$. 
If the orbit length of Melnikov functions $M^\gamma_j(t)$ of the Poincaré map of this deformation along $\gamma$ is at most $j-k$, for $1\le j\le n$, then   
the orbit length of  Melnikov functions $M^\gamma_j$ is at most $j-k$, for all $j\in\N$.

In particular, this holds if the Melnikov functions $M^\gamma_j(t)$ vanish identically, for $1\le j\le n$.

\end{customthm}

\begin{definition}
    \label{def:noether}
    We call the smallest $n_{H,\gamma}(k)$, verifying the conclusion of the theorem,  the \emph{Noetherianity index} of $H,\gamma,k$ and denote it $\nu_{\scriptscriptstyle  H,\gamma}(k)$.
\end{definition}

\begin{remark}
    The main information determining the Noetherianity index is topological, carried by the original system $dH$ and $\gamma$ and not by the deformation $\eta$. Instead of a polynomial deformation $\eta$, one can also take $\eta$ given by an entire form and  everything works exactly in the same way, with the same bounds. 
\end{remark}

\begin{remark}
    Note that, by Fran\c coise's algorithm \cite{F,G05,FP}, if $M_j$ is a Melnikov function of order $j$, then $M_j$ is given by an iterated integral of length at most $j$. The theorem shows that this length can be reduced arbitrarily, if the first nonzero Melnikov function is of a sufficiently high index.  Our result applies to any Melnikov function and not only the first non-zero Melnikov function. 
    
    Note, however, that we work with the orbit length of iterated integrals (see Definition \ref{def:length/orbit}) instead of the usual length. In Example \ref{exm:square}, we give a system for which we calculate the  index $n_{H,\gamma}(k)$.
     The same example shows that the orbit length of its Melnikov functions does not coincide with the usual length. It seems essential, for our results, to work with the orbit length, instead of the usual length.

    We stress that the Noetherianity index  $\nu_{H,\gamma}$ holds for any polynomial $\eta$. Such a bound could be lower for some particular families $\eta$ (e.g. polynomial forms of a given degree).

\end{remark}

In order to prove the main result, Theorem \ref{thm:main1}, we prove a structure Theorem~\ref{Theo.M-sigma_k}:
\begin{customthm}{B}\label{Theo.M-sigma_k}
 For each $\ell\ge1$, and $\sigma_k\in\OO_k$, $k\in \N$, the Melnikov function $M_\ell^{\sigma_k}$ can be expressed as a polynomial in 
 $M_j^{\delta_i}$, $1\leq j\leq \ell-(k-1)$ and its derivatives of order at most $\ell-(k-1)$. This polynomial can be chosen to be independent of $\eta$.
\end{customthm}

This theorem expresses Melnikov functions on a cycle in the orbit as a differential polynomial in  some basic Melnikov functions. The key fact is that the set of these functions needed to describe restrictions of $M_\ell$ to $\OO_k$ depends only on $\ell-k$ basic Melnikov functions.  Our result makes the results of Fran\c coise and Pelletier \cite{FP} more precise, in particular with respect to the position of the cycle in the filtration $\OO_i$. 

Each one of these theorems has its universal counterpart: Theorem   \ref{thm:abstract2.3} and Theorem \ref{thm: thm universal} expressed in abstract differential algebra without the appearance of the deformation form $\eta$. The original theorems are then obtained from the corresponding universal theorems by evaluation for any form $\eta$.

\section{Idea of the proof of Theorem \ref{thm:main1}}\label{sect:Idea}
Our key argument is the following theorem:
\begin{theorem}[Ritt-Raudenbush basis theorem, {\cite{[R]}, \cite{K57}}]\label{thm:Ritt-Raudenbush}
Let $\mathcal{A}$ be a finitely generated differential algebra. Then, every strictly increasing chain of radical differential ideals is finite.
\end{theorem}
We stress that $\mathcal{A}$ is not necessarily differentially Noetherian: the radicality requirement above is necessary. For example, the ascending chain of differential ideals $I_1\subset I_2\subset\cdots \subset I_n\subset\cdots$, where $I_n=<x^2,(x')^2,\dots,(x^{(n)})^2>$, where $x^{(n)}=\frac{d^n x}{dt^n}$, does not stabilize.

The  strategy of the proof is based on the following pairing matrix:
\begin{equation}\label{table-M_i}    
    \centering
   \begin{array}{c|ccccc}
           &\OO _1  & \OO _2 & \OO _3&\OO_4&\cdots\\
           \hline%
        M_1 & M^1_1 & 0 & 0 &0&\cdots \\
         M_2& M^1_2 & M^{2}_2 &  0&0 &\cdots\\
         M_3&  M^1_3&  M^{2}_3& M^{3}_3 &0 &\cdots\\
         M_4&  M^1_4&  M^{2}_4& M^{3}_4 & M^{4}_4&\cdots\\
        \vdots& \vdots & \vdots & \vdots &\vdots &\ddots\\
    \end{array},
\end{equation}
where $M_j^{i}$ means the restriction of the Melnikov function of order $j$, to the quotient $\OO_i=\frac{\OO\cap L_i}{\OO\cap L_{i+1}}$. Note that the paring matrix is lower triangular, because $M_j$ is an iterated integral of length at most $j$, and hence vanishes on $\OO_i$, for $i>j$.

The table should be understood in the following sense: we are interested in the successive annihilation of the diagonals. First note that the elements in the first diagonal are well defined since $M_j$ vanishes on $\OO\cap L_{j+1}$ due to the length of $M_j$. Now, elements in the second diagonal are not well defined since they depend on the representative of the cycle in the quotient. However, assuming the vanishing of the first diagonal they become well defined. The same holds for any successive diagonal.

For each $i$, choose a basis $\beta_i:=\{\sigma_{i,\alpha}\}$ of $\OO_i$. Assuming the vanishing of the first $(r-1)$ diagonals, for $j=i+r-1$, the functions $M_j^{\sigma_{i,\alpha}}(t)$, which belong to the $r$-th diagonal, are well-defined, for all of the elements of the basis of $\OO_i$, as $M_j$ identically vanishes on $\OO \cap L_{i+1}$.

In order to prove Theorem \ref{thm:main1}, we want to use the Ritt-Raudenbush theorem. We hence need a universal differential algebra. This algebra is given by Theorem \ref{Theo.M-sigma_k}.

Theorem~\ref{Theo.M-sigma_k} claims that, for all $i\in\N$, the functions $M_{i+r-1}^{\sigma_{i,\alpha}}(t)$ are differential polynomials in some fixed set of functions, depending on $H$, $\eta$, and $r$ only:
\begin{equation}
   M_{i+r-1}^{\sigma_{i,\alpha}}(t)=m_{r,i,\sigma_{i,\alpha}}(M^{\delta_i}_j),\quad 1\le i\le m,\,1\le j \le r,
\end{equation}
where $\{\delta_i\}_{i=1}^m$ is a basis for the first homology group of the regular fiber $\{H=t\}$.
Moreover, these differential polynomials $m_{r,i,\sigma_{i,\alpha}}$  are universal in the sense that they depend only on $r$, $H$ and $\sigma_{i,\alpha}$, i.e. do not depend on $\eta$. In other words, the function $M_{i+r-1}^{\sigma_{i,\alpha}}(t)$ is a result of the evaluation of some independent of $\eta$ differential polynomial in $mr$ variables on an $mr$-tuple of germs of holomorphic functions defined by $\eta$ and independent of $i$ and $\sigma_{i,\alpha}$. In this way, the dependency of $M_{i+r-1}^{\sigma_{i,\alpha}}(t)$ on the cycle and the foliation is decoupled from its dependency on the perturbation.

\begin{example} By \cite[Proposition 3.4]{MNOP2} every $M_j^{\sigma_j}$, with $\sigma_j\in\OO_j$, is given by the Wronskian of abelian integrals $\int_{\delta_i}\eta$. These Wronskians define the differential polynomials given in Theorem  \ref{Theo.M-sigma_k}, for the case $k=1$.

   In particular, for a cycle ${\sigma_2}=[\delta_1,\delta_2]\in L_2$, the Melnikov function $M_2^{\sigma_2}(t)$ is given by the Wronskian 
\begin{equation*}
    M_2^{\sigma_2}(t)=W\left(\int_{\delta_1}\eta,\int_{\delta_2}\eta\right),
\end{equation*}
which is the formal differential polynomial $$W(m_1(\delta_1),m_1(\delta_2))=m_1(\delta_1)m_1(\delta_2)'-m_1(\delta_1)'m_1(\delta_2)$$ evaluated at the tuple $\left(\int_{\delta_1}\eta, \int_{\delta_2}\eta\right)$. The Wronskian $W$ is a differential polynomial $W(m_1(\delta_1),m_1(\delta_2))$  lying in a free differential $\C$-algebra $\mathcal{A}_1$ generated by abstract symbols $m_{1}(\delta_i)$, and their formal derivatives, and this differential polynomial $W$ is the same for any perturbation $\eta$. 

The functions $M_1^{\delta_i}=\int_{\delta_i}\eta$ are obviously dependent on $\eta$, and this is the only place where the dependence of $M_2^\sigma(t)$ on $\eta$ comes from.

More generally,  for any $\sigma_{j}\in\OO_{j}$, the corresponding Melnikov functions $M^{\sigma_j}_j(t)$ can be expressed as the evaluations of a polynomial
$m_j(\sigma_j)\in\A_1$ at the tuple $(M_1^{\delta_1},\dots,M_1^{\delta_m})$. 

Given a commutator $\sigma_j$ of length $j$, the corresponding differential polynomial $m_j(\sigma_j)$ is given by the iterated Wronskian, respecting the structure of the commutators in $\sigma_j.$ The product of commutators corresponds to a sum of the corresponding iterated Wronskians. Thus, the polynomial $m_j(\sigma_j)$ is independent of $\eta$.
Theorem~\ref{Theo.M-sigma_k} generalizes the above study, for $k\geq1$. 
\medskip

Now, we give a sketch of the proof of Theorem \ref{thm:main1}, for $k=1$. 
The Ritt-Raudenbush basis theorem claims that there exists a finite tuple of polynomials $m_{j'}(\sigma_{j'\ell})\in \A_1$, with $\sigma_{j'\ell}\in\OO_{j'}$, with the property that  any $m_j(\sigma_j)$, with $j>1$ and $\sigma_{j}\in\OO_{j}$, raised to some power, is a polynomial combination of these polynomials
\begin{equation}\label{eq:RR relation in example}
\left(m_j(\sigma_j)\right)^{d_j}=
\sum_{1\le j'\le n}\sum_{1\le \ell'\le \ell'_{j'}}\sum_{\alpha\ge 0}
q_{j'\ell'\alpha}\left(m_{j'}(\sigma_{j'\ell'})\right)^{(\alpha)}.
\end{equation}
Here $\left(m_{j'}(\sigma_{j'\ell'})\right)^{(\alpha)}\in\A_1$ are derivatives of $m_{j'}(\sigma_{j'\ell'})$ of order $\alpha$ and $q_{j'\ell'\alpha}\in \A_1$.

The evaluation at the tuple $(M_1^{\delta_1},\dots,M_1^{\delta_m})$ will preserve the relation~\eqref{eq:RR relation in example} and will send $m_{j'}(\sigma_{j'\ell'})$ to $M_{j'}^{\sigma_{j'\ell'}}$. 
Thus, if $M_{j'}^{\sigma_{j'\ell'}}$ vanish identically for a given $\eta$, then the evaluation of $m_j(\sigma_j)$ (i.e., $M_j^{\sigma_j}$) will vanish as well. In other words, the orbit lengths of $M_j^{\sigma_j}$ are at most $j-1$.
\medskip
\end{example}

The  differential algebra $\A_r$, for $1\leq r\leq k$, generated by $\{m_{j}(\delta_i)\}_{1\le i\le m, 1\le j\le r}$, plays the same role: for every $\sigma\in \OO_{\ell-r+1}$, there is a polynomial $m_{\ell}(\sigma)\in\A_r$ such that
the evaluation  
\begin{equation}\label{eq:ev}
\ev_\eta:\A_r\rightarrow\A^{\eta}_r, \ \ \ev_\eta(m_{j}(\delta_i))=M^{\delta_i}_j
\end{equation}
sends $m_{\ell}(\sigma)\in\A_r$ to $M^{\sigma}_\ell$.
Here, $\Hol$ denotes the differential algebra of the germs of holomorphic functions $(\C,0)\to\C$ and $\A^{\eta}_r\subset \Hol$ is its differential subalgebra generated by the first $r$ Melnikov function $M^{\delta_i}_j$ along the basis $\{\delta_i\}$,  $1\le i\le m$ and $1\le j \le r$.  

Using the Ritt-Raudenbush Theorem \ref{thm:Ritt-Raudenbush} and the same argument as above, we will prove Theorem~\ref{thm:main1}.

\subsection{Structure of the article}
In Section \ref{sect.Mel.on.Comm}, we study properties of Melnikov functions corresponding to the first return map of the deformation $dH+\epsilon\eta=0$, with respect to the commutators $\sigma_\ell\in\OO_\ell$ belonging to the orbit of $\gamma$ and we introduce the \emph{universal holonomy}. We also define two algebraic invariants: \emph{length} and \emph{triangularity}. 

In Section \ref{Sec:Universal-Noetherian} we establish the universal counterparts 
(Theorem \ref{thm:abstract2.3} and \ref{thm: thm universal}) of Theorems \ref{thm:main1} and \ref{Theo.M-sigma_k}, we prove them and using these universal results, via the evaluation, we prove the Theorems \ref{thm:main1} and \ref{Theo.M-sigma_k} themselves.

In Section \ref{sect.Diff.Op.Hol} we study Theorem \ref{Theo.M-sigma_k} from another point of view, using differential operators and establish an explicit isomorphism between the \emph{universal holonomy group} $\mathfrak{g}$ and the group of lower triangular Toeplitz matrices $\mathcal{T}$ . It gives a matricial interpretation for the triangularity index. 

In Section \ref{Sec:Examples}, we compute the universal Noetherianity index $n_{H,\gamma}$ for some examples. 

In Section \ref{Sec:Lie-alg}, we show that for the principal diagonal the free differential algebra $\A_1$ can be replaced by a much smaller free Lie algebra $\mathfrak{a}\subset\A_1$. On the first diagonal, we also show the optimality of the universal Noetherianity index  for perturbations by holomorphic form $\eta$.
Finally, in Section \ref{Sec:Prospective} we give some prospective and open problems.

\section{Melnikov functions on commutators}\label{sect.Mel.on.Comm}

In this Section we develop some techniques to prove Theorem \ref{Theo.M-sigma_k}. This theorem gives a structure theorem about the form of Melnikov functions along a loop $\sigma_k\in L_k$ in function of the position of the loop $\sigma_k$, with respect to the filtration given by the lower central sequence $L_1\supset L_2\supset\ldots$. It improves the result of Fran\c coise-Pelletier \cite{FP} about the form of the Melnikov function, if $k>1$.

We give two proofs of Theorem~\ref{Theo.M-sigma_k}. The first one, given in Section~\ref{Sec:Universal-Noetherian}, deals with the algebra of series compositions directly. We also introduce the notion of the \emph{triangularity index}. The second proof is given in Section~\ref{sect.Diff.Op.Hol} and studies the linear differential operators given by precomposition with holonomy maps. Both are ultimately based on the Taylor formula. While the first one is shorter, we think that the second proof explains better the phenomena involved, and the notion of the triangularity index  in particular.\\


Let $\{\delta_i(t)\}_{i=1,\dots,m}$ be a basis of vanishing cycles of the first homology group of a regular fiber $\{H=t\}$. Consider the first return map, i.e. the holonomy, with respect to the foliation $\mathcal{F}_\epsilon$, along the loop $\delta_i$, for any $i$.
For any $\delta\in\pi_1(\{H=t\},p_0)$, we consider its holonomy map 
\begin{equation}\label{eq:basis holonomy2}
    P^{\delta}_\epsilon(t)=t+\epsilon M^{\delta}_1(t)+\epsilon^2M_2^{\delta}(t)+\cdots.
\end{equation}
In particular, for the generators $\delta_i$ of $\pi_1(\{H=t\},p_0)$, we have
\begin{equation}\label{eq:basis holonomy}
    P^{\delta_i}_\epsilon(t)=t+\epsilon M^{\delta_i}_1(t)+\epsilon^2M_2^{\delta_i}(t)+\cdots.
\end{equation}

Recall that
\begin{equation}\label{eq:anti homo}
    P^{\alpha\beta}_\epsilon(t)=P^{\beta}_\epsilon\circ P^{\alpha}_\epsilon(t),
\end{equation}
so the holonomy map is an anti-representation.

\subsection{Holomorphic holonomy}\label{sec:proof1}
Let $\Hol$ be the differential algebra of germs of holomorphic functions $\phi:(\C,0)\rightarrow \C$, and let  $\Hol[[\epsilon]]$ be the differential   algebra, with respect to $\frac{\partial}{\partial t}$, of formal series $\sum_{j\ge 0} \epsilon^j\phi_j$, $\phi_j\in\Hol$.  Let $\mathcal{B}=\id+\epsilon \Hol[[\epsilon]]\subset \Hol[[\epsilon]]$.
    
     For $g\in  \Hol[[\epsilon]]$, the  powers $(\ep g)^n$ are well defined. Thus, for $f\in\Hol$,  we can define its composition with $\id+\ep g\in\mathcal{B}$ via the Taylor expansion:
    \begin{equation}\label{eq:Taylor in B}
        f(\id+\epsilon g)=f + \epsilon f'g+\frac{\epsilon^2}{2} f''g^2+\cdots\in  \Hol[[\epsilon]].
    \end{equation}
    For any $F=\id+\sum \ep^j f_j, G=\id+\ep g\in\mathcal{B}$,   the composition $F\circ G$ is therefore given as 
    \begin{equation}\label{eq:Taylor in B 2}
        (\id+\sum \ep^j f_j)(\id+\ep g)= \id+\ep g +\sum_{j>0}\ep^j f_j(\id+\ep g)\in \mathcal{B}.
    \end{equation}
\begin{lemma}
    The set $\mathcal{B}$ is a group under composition.   
\end{lemma}    
Using \eqref{eq:basis holonomy}, we define the holonomy map  
\begin{equation}\label{eq:hol eta}
    \pe:\pi_1\left(H^{-1}(t_0)\right)\rightarrow \mathcal{B}\subset Hol[[\epsilon]], \quad \pe(\delta)=P_\epsilon^{\delta}(t).
\end{equation} This is an anti-homomorhism of groups by \eqref{eq:anti homo}, so the map $\pe$ is completely defined by $\pe(\delta_i)$.

    For a $\delta\in \pi_1\left(H^{-1}(t_0)\right)$, write  
    \begin{equation*}
        \pe(\delta)=\id+\sum \ep^j M_j^{\delta}.
    \end{equation*}

\subsection{Universal holonomy}\label{subsec: universal holonomy}
Here we first define a \emph{universal holonomy},  a map $\scp$ from $\pi_1(\{H^{-1}(t_0)\}, p_0)$ to a suitable group $\mathcal{G}$ defined below. This lifting  depends on $\pi_1(\{H^{-1}(t_0)\}, p_0)$ only, and not on the perturbation $\eta$. Any holonomy $\pe$ corresponding to a perturbation $\eta$ equals the composition of $\scp$ with an evaluation map $\ev_\eta$.

The main part of this section is a proof of Theorem~\ref{thm: thm universal} , which is an analogue of Theorem~\ref{Theo.M-sigma_k} for the universal holonomy. This implies Theorem~\ref{Theo.M-sigma_k}, for any $\eta$, by the universality property (see Lemma \ref{lem:universality}).
\medskip

Define the free differential algebra $\mathcal{A}$ with a unit generated by $m_{j}(\delta_i)$, $1\le i\le m$, $j\ge 1$, and 
denote by $\A_r\subset\A $ the free differential algebra generated by $m_{j}(\delta_i)$, $1\le i\le m$, $1\le j\le r$. Let $\mathcal{A}[[\epsilon]]$ be the differential  algebra of formal series $\sum_{j\ge 0} \epsilon^j\phi_j$, $\phi_j\in\mathcal{A}$.
Define $\mathfrak{m}\subset\mathcal{A}$ to be the maximal differential ideal of $\mathcal{A}$. Let $\mathfrak{g}=\bigoplus_{j>0}\ep^j\mathfrak{m}\subset\mathcal{A}[[\epsilon]]$.

We define the composition on the set $\mathcal{G}=\{\id\}\oplus\mathfrak{g}$ of formal series, using Taylor series by:
\begin{equation*}
    \ep^j f_j\circ(\id+g)=\ep^j\left(f_j+\sum_i\frac{f_j^{(i)}}{i!}g^i\right), \quad t\circ(\id+g)=\id+g
\end{equation*}
and extend by linearity. 
For any $f,g\in\mathfrak{g}$, define the composition as 
\begin{equation*}
    f\circ(\id+g)=(\id+f)\circ(\id+g)-(\id+g).
\end{equation*}
\begin{lemma}
    $\mathcal{G}$ is a group under the composition defined above.
\end{lemma}
\begin{lemma}\label{lem:Gk small}
    Denote by $\mathcal{G}_k$ the $k$-th term of the lower central sequence of $\mathcal{G}=\mathcal{G}_1$. Then $\mathcal{G}_k\subset \{\id\}\oplus\bigoplus_{j\ge k}\ep^j\mathfrak{m}$.
    \end{lemma}
    \begin{proof}
        $\mathcal{G}_k$ is generated by commutators of length $\ge k$ of the basic elements. Thus, the claim follows from 
        \begin{equation}\label{eq:vector}
            [\id+\ep f, \id+\ep^ig]=\id+\ep^{i+1}(f'g-g'f)+\dots.
        \end{equation}
    \end{proof}

Recall that, as $H^{-1}(t_0)$ is a smooth affine curve, the group $\pi_1(\{H^{-1}(t_0)\}, p_0)$ is a free group $\mathbb{F}_m$, generated by $\delta_1,\ldots,\delta_m$. Define a homomorphism, further called \emph{universal holonomy}, by  
\begin{equation}\label{eq: universal mono def}
    \scp: \mathbb{F}_m\cong\pi_1(\{H^{-1}(t_0)\}\,, p_0)\to \mathcal{G},\quad \scp\left(\delta_i\right)=t+\sum_{j>0}\ep^jm_{j}(\delta_i)\in \mathcal{G}. 
\end{equation}
We denote by $m_j(\delta)$ the components of the map $\scp$:
\begin{equation*}
    \scp(\delta)=t+\sum_{j>0}\ep^jm_j(\delta).
\end{equation*}

For a given $\eta$, define the evaluation homomorphism of differential algebras $\ev_\eta: \mathcal{A}\to \Hol$ by $\ev_\eta(m_{j}(\delta_i))=M_j^{\delta_i}$, which we extend to $\ev_\eta: \mathcal{A}[[\epsilon]]\to \Hol[[\epsilon]]$. We denote by 
\begin{equation}\label{eta evaluation}
    \A_r^\eta=\ev_\eta(\A_r)\subset \Hol\,.
\end{equation}
 
\begin{lemma}\label{lem:universality}
    The map $\scp$ has the following universal property:  $\pe=\ev_\eta\circ\scp$ for any $\eta$.
\end{lemma}

\medskip

For a differential polynomial $\phi\neq0$ in $\mathfrak{m}$, we define its \emph{universal length} $\val{\phi}$ as the smallest $r$ such that $\phi$ is a differential polynomial in $m_{j}(\delta_i)$, with $1\le j\le r$ only,
\begin{equation*}
    \val{\phi}=\min\{r|\phi\in \A_r\}.
\end{equation*}
\begin{remark}\label{rem:positive length}
    Note that $\val{\phi}\ge 1$, for $\phi\neq0$, as $\mathfrak{m}$ doesn't contain non-zero constants. 
\end{remark}

For a differential polynomial
$f=\sum_{j>0}\ep^jf_j$, we define its \emph{triangularity index}  
\begin{equation*}
    \de{f}=\inf_j\{j-\val{f_j}+1\}.
\end{equation*}
In our application we only consider $f$ such that $\de{f}\ge 1$. We put $\de{\id+f}=\de{f}$ and define $\val{0}=-\infty$.
\begin{remark}
In Section \ref{sect.Diff.Op.Hol}, we study holonomies 
appearing in \eqref{eq:deformation} and differential operators associated with these holonomies. They motivate the introduction of the universal holonomy group  $\mathfrak{g}$ in this section. 
We give matrix operators associated with these differential operators and show that they are lower triangular Toeplitz matrices $\mathcal{T}$. 
In Subsection \ref{sub7.3}, \eqref{eq:iso}, we establish an isomorphism  between the universal holonomy group $\mathfrak{g}$ and the group of Toeplitz matrices $\mathcal{T}$.
We show moreover, in Proposition \ref{prop:bridge}, that the isomorphism transports the notion of \emph{triangularity} $\tau$
from the universal holonomy group to the corresponding group of Toeplitz matrices. 
This explains the terminology of triangularity, comming from the number of diagonals vanishing in the matrix representation. 
\end{remark}

\section{Universal structure and Noetherianity theorems}\label{Sec:Universal-Noetherian}

In this section we develop two universal theorems Theorem \ref{thm: thm universal} and Theorem \ref{thm:abstract2.3} corresponding to Theorem \ref{Theo.M-sigma_k} and Theorem \ref{thm:main1}, respectively. 
The first one is an abstract structure theorem for Melnikov functions in the differential algebra $\mathcal{A}$ and the second one gives a universal Noetherianity index in the same algebra $\mathcal{A}$.

\begin{customthm}{B'}\label{thm: thm universal}
   Let $\delta\in L_\ell$, where $L_\ell$ is the $\ell$-the element of the lower central series of $\mathbb{F}_m\cong\pi_1(\{H^{-1}(t_0)\}, p_0)$, and let $\scp(\delta)=t+\sum_{j>0}\ep^jm_j(\delta)$. Then $\val{m_j(\delta)}\le j-\ell+1$, so $m_j(\delta)\in \A_{j-\ell+1}$.
\end{customthm}

\medskip
In order to prove Theorem~\ref{thm: thm universal}, we establish some basic properties of the length and the triangularity index.
\begin{lemma}\label{lem:de for product of monomials}
     $\de{\epsilon^{i+j}f_ig_j}\ge\de{\epsilon^{i}f_i}+\de{\epsilon^{j}g_j}$.
\end{lemma}
\begin{proof}
    \begin{align*}
      \de{\epsilon^{i+j}f_ig_j}=i+j-\val{f_ig_j}+1=i+j-\max\{\val{f_i},\val{g_j}\}+1\\
        =(i-\val{f_i}+1)+(j-\val{g_j}+1) +\min\{\val{f_i},\val{g_j}\}-1\\
        =\de{\epsilon^{i}f_i}+\de{\epsilon^{j}g_j} +\min\{\val{f_i},\val{g_j}\}-1\\
        \ge\de{\epsilon^{i}f_i}+\de{\epsilon^{j}g_j},
    \end{align*}
    as $\val{f_j},\val{g_j}\ge  1$, by Remark~\ref{rem:positive length}.
\end{proof}
\begin{lemma}\label{lem:de with add and mult}
    For $f,g\in\mathcal{B}$, we have
    \begin{enumerate}
        \item $\de{f+g}\ge \min\{\de{f},\de{g\}}$, 
        \item $\de{fg} \ge\de{f}+\de{g}$. In particular, $\de{g^n}\ge n\de{g}$.
    \end{enumerate}
\end{lemma}
\begin{proof}
Both assertions follow from considering the monomials, using Lemma~\ref{lem:de for product of monomials}, in the second case.
\end{proof}

\begin{lemma}\label{lem:composition triangularity index s} 
For $f,g\in\mathcal{B},$ 
    \begin{equation}
        \de{f\circ(\id+g)-f}\ge\de{f}+\de{g}.
    \end{equation}
\end{lemma}
\begin{proof}
    By Taylor decomposition,
    \begin{equation*}
        f\circ(\id+g)-f=\sum_{i>0} \ep^i\left(f_i\circ(\id+g)-f_i\right)=\sum_{i>0} \ep^i\sum_{j> 0}\frac{f_i^{(j)}}{j!}g^j,
    \end{equation*}
    so it is enough to prove that $\de{\ep^i f_i^{(j)} g^j}\ge \de{f}+\de{g}$. Clearly, $\val{ f_i^{(j)}}=\val{f_i}$, so 
    \begin{equation*}
\de{\ep^i f_i^{(j)} g^j}=\de{\ep^i f_i g^j}\ge \de{f}+j\de{g}\ge \de{f}+\de{g}.
\end{equation*}
\end{proof}

\begin{corollary}\label{cor:De with comp}\hfill
   \begin{enumerate}
       \item $\de{f\circ G}\ge\de{f}$
       \item $\de{F\circ G}\ge\min\{\de{F},\de{G}\}.$
   \end{enumerate} 
\end{corollary}
\begin{proof}
Write
$$f\circ G=(f\circ G-f)+f$$
or, respectively,
$$F\circ G= G+f\circ G =(f\circ G-f)+f+G$$ and use the previous Lemma and Lemma~\ref{lem:de with add and mult}.
\end{proof}

\begin{lemma}
    $\de{F^{-1}}=\de{F}$, $\de{F^{\circ-1}}=\de{F}$.
\end{lemma}
\begin{proof}
    For $F^{-1}=t+\sum_{i>0}\ep^i\tilde{f}_i$,  the coefficients $\tilde{f}_i$ are polynomials in $f_1,...,f_i$, so $\val{\tilde{f}_i}\le\max_{i'\le i}\val{f_{i'}}$. Thus $\de{F^{-1}}\ge\de{F}$, and, by symmetry, we have the opposite equality as well.
    The same proof works for ${F^{\circ-1}}$.
\end{proof}

\begin{proposition}\label{prop:proposition on commutators}
    $\de{F\circ G \circ F^{-1}\circ G^{-1}}\ge\de{F}+\de{G}$.
\end{proposition}  
\begin{proof}
Denote as above $F=\id+f$, $G=\id+g$, $f,g\in\mathcal{B}$.

Note that $F\circ G \circ F^{-1}\circ G^{-1}=\id + (F\circ G - G\circ F)\circ\left(F^{-1}\circ G^{-1}\right)$, so we need to prove that 
\begin{equation}\label{eq:to prove for commutators}
    \de{(F\circ G - G\circ F)\circ\left(F^{-1}\circ G^{-1}\right)}\ge\de{F}+\de{G}.
\end{equation}
But \begin{align*}
    F\circ G - G\circ F=(\id+g+f\circ(\id+g))-(\id+f+g\circ(\id+f))\\=(f\circ(\id+g)-f)-\left(g\circ(\id+f)-g\right), 
\end{align*}
so $ \de{F\circ G - G\circ F}\ge\de{F}+\de{G}$, by Lemma \ref{lem:composition triangularity index s} and Lemma~\ref{lem:de with add and mult}.
    The claim now follows from  Corollary~\ref{cor:De with comp}(1).
\end{proof}

\begin{proof}[Proof of Theorem~\ref{thm: thm universal}]
    Indeed, $\val{m_j(\delta)}\le j-\de{\scp(\delta)}+1$. As $\delta\in L_k$, one can represent $\delta$ as a product of commutators $\sigma_i\in L_i$ of the cycles $\delta_i$ of lengths $\ge k$, $\delta=\prod \delta_i$.
    Thus $\scp(\delta)$ is a composition of commutators $\scp(\sigma_i)$ of $\scp(\delta_i)$ of lengths $\ge k$. By definition, $\de{\scp(\delta_i)}=1$. Thus, by Proposition~\ref{prop:proposition on commutators}, $\de{\scp(\sigma_i)}\ge k$ and therefore $\de{\scp(\delta)}\ge k$ by Corollary~\ref{cor:De with comp}, which implies the claim of the Theorem.
\end{proof}
\begin{proof}[Proof of Theorem~\ref{Theo.M-sigma_k}]
     Theorem~\ref{Theo.M-sigma_k} follows from Theorem~\ref{thm: thm universal} by evaluation $\ev_\eta$, using the universality property in Lemma~\ref{lem:universality}.
\end{proof}
\subsection{Proof of Theorem~\ref{thm:main1}}\label{sect.Diff-Alg}

For all $q\ge 1$, choose a finite collection  $\{\sigma_{q,q'}\}_{1\le q'\le d_q}\subset \OO\cap L_q$  generating $\frac{\OO\cap L_q}{\OO\cap L_{q+1}}$.

Let $c_{r,s}$ be a finite collection of polynomials
\begin{equation}\label{eq:Drs collection def}
    c_{k,s}=\{m_s(\sigma_{q,q'}), s-k+1\le q\le s\}, 
\end{equation}
and let $C_{k,s}=\cup_{s'\le s}c_{k, s'}$ and $C_{k}=\cup_s C_{k,s}$.
\begin{figure}[h]
    \centering
    \includegraphics[width=0.5\linewidth]{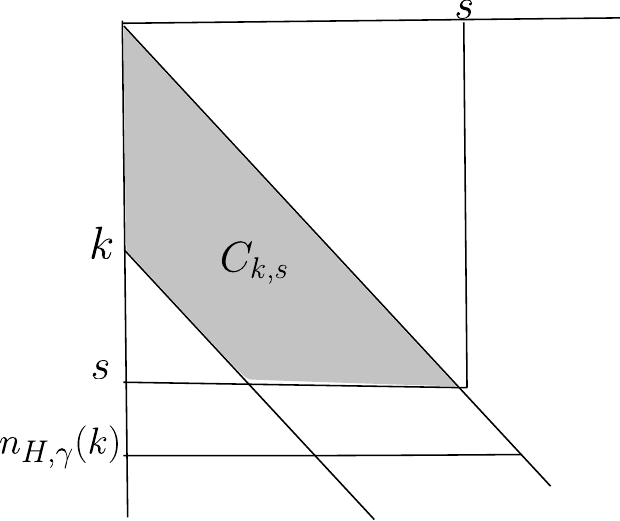}
   \caption{The first $k$ diagonals truncated at order $s$ (which equals to the length $s$).}
   \label{table:Cr-Dr}
\end{figure}
\begin{lemma}\label{lem:Crs generation}
    Let $\delta\in \OO\cap L_q$ and let $s-k+1\le q\le s$. Then $m_j(\delta)$ can be expressed as a differential polynomial of elements of $C_{k,s}$.
\end{lemma}
\begin{proof}
    Write $\delta=\left(\prod_i\sigma_{q_i,{q'}_i}\right) \delta_{s+1}$, with ${q\le q_i\le s+1}$ and $\delta_{s+1}\in L_{s+1}$. Note that, by Lemma~\ref{lem:Gk small},
    \begin{equation*}
        \scp(\delta_{s+1})\subset \mathcal{G}_{s+1}\subset \{id\}+O(\ep^{s+1}).
    \end{equation*} Therefore, applying $\scp$, we obtain  the composition equality
\begin{equation*}
    \scp(\delta) =\left(\id+O(\ep^{s+1})\right)\circ\left(\prod _i\scp(\sigma_{q_i,{q'}_i})\right) =\prod \scp(\sigma_{q_i,{q'}_i})+O(\ep^{s+1}).
\end{equation*}
Developing this equality using composition rules and comparing the coefficients of $\ep^s$ on both sides proves the statement of the Lemma.
\end{proof}

By Theorem~\ref{thm: thm universal}, $C_{k}\subset \A_k$. Denote by $\mathcal{C}_{k,s}, \mathcal{C}_k$ the differential ideals of $\A_k$ generated by $C_{k,s}, C_{k}$, respectively. 
\begin{corollary}
    \begin{equation*}
    \mathcal{C}_{k,s}=\left<m_j(\delta), \delta\in \OO\cap L_{s'}, s'\le s, j\le s'+k-1\right>.
\end{equation*}
\end{corollary}

\begin{proof}
Indeed, the inclusion follows from  Lemma~\ref{lem:Crs generation}, and the opposite inclusion follows since the right side contains generators of  
 $C_{k,s}$.\end{proof}

\begin{customthm}{A'}\label{thm:abstract2.3}
    For each $k\in\N$, there exists $n=n_{H,\gamma}(k)\in\N$ such that, for all $\delta\in\OO\cap L_\ell$, $\ell\ge 0$,   some power of $m_j(\delta)\in \A_k$, $ j\le \ell+k-1$, is a given by a differential polynomial in  $m_{j'}(\sigma_{q,q'})$, where  $1\le q\le n$  and $q\le j'\le q+k-1$.
\end{customthm}

\begin{proof}[Proof of  Theorem \ref{thm:abstract2.3}]
By the very definition, $C_{k,s}\subset C_{k, s+1}$. Thus we have an increasing chain of radical ideals
$$\sqrt{\CC_{k,1}}\subseteq\sqrt{\CC_{k,2}}\subseteq\cdots\subseteq\sqrt{\CC_{k,h}}\subseteq\cdots\le \A_k.$$

By Ritt-Raudenbush Theorem \ref{thm:Ritt-Raudenbush}, this increasing chain of ideals stabilizes. That is, there is an $n\in\N$, such that,
\begin{equation}\label{eq:n-chain}
    \sqrt{\CC_{k,1}}\subseteq\sqrt{\CC_{k,2}}\subseteq\cdots\subseteq\sqrt{\CC_{k,n}}=\sqrt{\CC_{k,n+1}}=\cdots.
\end{equation}    
 This number $n=n_{H,\gamma}(k)$ verifies the conclusions of Theorem \ref{thm:main1}.

Indeed, this means that for any $\delta\in L_s\cap \OO$, $s\ge n$, and $j\le s+k-1$ there is a $d=d(\delta, j, k)$ such that
\begin{equation}\label{eq:RR in universal}
    m_j(\delta)^d=\sum_{1\le q\le n}\sum_{q'}\sum_{q\le j'\le q+k-1}\sum_{\alpha\ge 0} c_{q,{q'},j',\alpha} m_{j'}(\sigma_{q,{q'}})^{(\alpha)}, 
\end{equation}
where  $c_{q,{q'},j',\alpha}\in \A_k$ and $m_{j'}(\sigma_{q,{q'}})^{(\alpha)}$ are derivatives of $m_{j'}(\sigma_{q,{q'}})$ of order $\alpha$. Note that this remains true for $s< n$ by Lemma~\ref{lem:Crs generation}.
\end{proof} 
\begin{proof}[Proof of Theorem ~\ref{thm:main1}]
Let $\eta$ be a differential form as in \eqref{eq:deformation}. Recall that $\ev_\eta: \A_k\to\A_k^\eta\subset\Hol$ is a homomorphism of differential algebras and 
$\ev_\eta\left(m_{j'}(\sigma_{q,{q'}})\right)=M_{j'}^{\sigma_{q,{q'}}}$. Let $k$ be given and let $n$ be given by Theorem \ref{thm:abstract2.3}. 
Then $ev_\eta\left(m_{j'}(\sigma_{q,{q'}})\right)=0,$ for $m_j'(\sigma_{q,{q'}})\in\mathcal{C}_{k,n}$ and this implies the vanishing of $M_j^{\sigma_{q,{q'}}}=ev_\eta\left(m_{j'}(\sigma_{q,{q'}})\right)=0,$ for $m_{j'}(\sigma_{q,{q'}})\in\mathcal{C}_k$, together with all their derivatives. Thus, by \eqref{eq:RR in universal}, 
\begin{equation*}
    M_j^{\delta}\equiv0,
\end{equation*}
for all $\delta\in\OO\cap L_\ell$, where $\ell\ge 0$ and $j\le \ell-k+1$.
It then follows from the definition of the orbit length that $M_j^{\gamma}$ is of orbit length less than or equal to $j-k$ for all $j\ge 0$.
\end{proof}

 The number $n=n_{H,\gamma}(k)$  can be computed in a different way. Let $$D_{r,s}:= < m_{r}(\sigma_{1}),m_{r+1}(\sigma_{2}),\dots,m_{r+s-1}(\sigma_s) >$$ be the $r$-th diagonal truncated at the $s$-th row as in Figure \ref{table:Cr-Dr}. Then for each $r$, we have  an ascending chain of radical ideals that stabilizes at some index $n_r\in\N$.
 We choose this index to coincide with the row in the table \eqref{eq:Drs collection def}. More precisely, putting
 \begin{equation}\label{eq:D_r-n_r}
    \sqrt{D_{r,1}}\subseteq \sqrt{D_{r,2}}\subseteq\cdots\subseteq\sqrt{D_{r,n_r-r+1}}=\sqrt{D_{r,n_r-r+2}}=\cdots,
     \end{equation}
 we get $\sqrt{D_{r,n_r-r+1}}=\sqrt{< m_{r}(\sigma_{1}),m_{r+1}(\sigma_{2}),\dots,m_{n_r}(\sigma_{n_r-r+1}) >}$.

Similarly, for any one form $\eta$, we define the ideals obtained by evaluation
$$
D_{r,s}^\eta=\ev_\eta(D_{r,s}).
$$
In analogy with \eqref{eq:D_r-n_r}, we define the stabilization index $\nu_r^\eta$, for the sequence of radical ideals $\sqrt{D_{r,s}^\eta}$.
We put
\begin{equation}\label{eq:nu_r}
\nu_r=\max_\eta{\nu_r^\eta}.
\end{equation}

  \begin{definition}\label{def:n_r}
      We call the stabilization index $n_r$, defined in \eqref{eq:D_r-n_r}, the \emph{universal Noetherianity index} of the $r$-th diagonal $D_r$.
      \end{definition}

 Then, by the construction, for each $k\in\N$, the  number 
 given by 
 \begin{equation}\label{eq:n}
     n=n_{H,\gamma}(k)=\max\{n_r\colon 1\le r\le k\}
 \end{equation}
 verifies the conclusion of Theorem \ref{thm:main1}.
 It is the smallest number verifying the conclusion of Theorem \ref{thm:abstract2.3}. We call the number $n_{H,\gamma}(k)$, the \emph{universal Noetherianity index}.

 Note that the \emph{Noetherianity index} index $\nu_{H,\gamma}(k)\leq n_{H,\gamma}(k)$ in Definition \ref{def:noether} is given by  
 \begin{equation}\label{eq:nu}
     \nu=\nu_{H,\gamma}(k)=\max\{\nu_r\colon 1\le r\le k\}.
 \end{equation}

\begin{remark}\label{rem:perturbations}
    The proof of  Theorem~\ref{thm:main1} doesn't use the fact that the holonomy homomorphism $P^{\delta}_\epsilon(t)$ is originated from a perturbation of a Hamiltonian foliation but only uses the group and ring properties of $\pi_1(H^{-1}(t_0),p_0)$  and $\Hol[[\ep]]$. Thus, it remains true (after a suitable reformulation) for analytic perturbations $P_\ep:\mathbb{F}_m\to\Hol[[\ep]]$ of the trivial homomorphism $P_0=\id:\mathbb{F}_m \to \{Id\}\in\Hol$.
\end{remark}

\section{Differential operators associated to holonomy}\label{sect.Diff.Op.Hol}

In this section, we study differential operators associated with holonomy maps.

\subsection{Differential operators associated to holonomy maps}\label{subsect.Diff.Op.Hol}

Consider the holonomy map with respect to a loop $\delta\in\pi_1(H^{-1}(t_0),p_0)$ 
\begin{equation}
    P^{\delta}_\epsilon(t)=t+\epsilon M^{\delta}_1(t)+\epsilon^2M_2^{\delta}(t)+\cdots.
\end{equation}

For $f\in\Hol$, we define \begin{equation}
    \p^\delta_\epsilon(f)(t):=f\circ P^\delta_\epsilon(t)\in\Hol[[\ep]].
\end{equation} Expanding $f(P_\epsilon(t))$, with respect to $\epsilon$ in $\epsilon=0$, we get

\begin{equation}\label{eq.Taylor-P_epsilon(f)}
\begin{array}{rl}
f(P_\epsilon(t))=&f(P_\epsilon(t))_{|\epsilon=0}+\epsilon\frac{\partial}{\partial \epsilon}(f(P_\epsilon(t)))_{|\epsilon=0}+\epsilon^2\frac{1}{2!}\frac{\partial^2}{\partial \epsilon^2}(f(P_\epsilon(t)))_{|\epsilon=0}+\cdots\\
\\
    =&f(t)+\epsilon f'(t) M^\delta_1(t)+\epsilon^2
    \frac{1}{2!}(f''(t)(M_1^\delta(t))^2+2f'(t)M^\delta_2(t))+\cdots.
    \end{array}
\end{equation}

That is,
\begin{equation}\label{eq.P_epsilon}
    \p_\epsilon(f)=f+\epsilon \left(M^\delta_1\frac{\partial}{\partial t}\right)(f)+\frac{\epsilon^2}{2!}\left((M^\delta_1)^2\frac{\partial^2}{\partial t^2}+2M^\delta_2\frac{\partial}{\partial t}\right)(f)+\cdots.
\end{equation}

Define differential operators $S_\ell^\delta:\Hol\rightarrow\Hol$ as follows:
\begin{equation}\label{eq.S_j^delta}
    S_\ell^\delta(f):=\frac{1}{\ell!}\frac{\partial^\ell}{\partial \epsilon^\ell}(f(P_\epsilon^\delta(t))_{|\epsilon=0}, \text{ for } \ell=1,2,...,\ \ \text{ where } S_0^\delta:=Id.
\end{equation}  
In particular \begin{equation}\label{eq.Op_S_ell}
    S_1^\delta=M^\delta_1(t)\frac{\partial}{\partial t},\  S_2^\delta=\frac{1}{2!}\left((M^\delta_1(t))^2\frac{\partial^2}{\partial t^2}+2M^\delta_2(t)\frac{\partial}{\partial t}\right),\dots
\end{equation}
Then \eqref{eq.P_epsilon} can be written as
\begin{equation}\label{eq.Op_P_epsilon}
    \p_\epsilon^\delta(f)=S^\delta_0(f)+\epsilon S_1^\delta(f)+\epsilon^2 S_2^\delta(f)+\cdots
\end{equation}
Define the vector space $$V_k=\Hol[[\epsilon]]/<\ep^{k+1}>.$$
We extend $\p_\epsilon^\delta$ to elements in $V_k$ by linearity:
\begin{equation}\label{eq.linear-extension-Vm}
    \p_\epsilon(\oplus \epsilon^jf_j)=\oplus \epsilon^j\p_\epsilon(f_j).
\end{equation}

Then 
$$\p^{\delta}_{\epsilon,k}:=Id+\epsilon S_1^\delta +\epsilon^2S_2^\delta+\cdots+\epsilon^kS_k,$$ with $S_\ell^\delta$, as in \eqref{eq.S_j^delta}. 
\begin{remark}
    In the sequel, for simplicity, we will omit the subindex $k$ in the notation of the operator $\p^{\delta}_{\epsilon,k}$ and the truncated diffeomorphism $P^{\delta}_{\epsilon,k}$.  We stress that the level of truncation $k$ is arbitrary and it depends on the order of the Melnikov function we want to study.
\end{remark}

 In particular, if we apply $\p_\epsilon^\delta$ to $f=id$ (i.e, $f(t)=t$), we recover $P_\epsilon^\delta$, as a diffeomorphism
$$\p_\epsilon^\delta(f)(t)=f(P_\epsilon^\delta(t))=P_\epsilon^\delta(t).$$

\begin{lemma}\label{Faa}
    The operator $S_{\ell}^{\delta}$, defined in \eqref{eq.S_j^delta}, can be expressed in the form

    \begin{equation}
        S_{\ell}^{\delta}=\sum_{m_1+2m_2+\cdots+\ell m_\ell=\ell}*\prod_{j=1}^\ell (M_j^{\delta}(t))^{m_j}\frac{\partial^{m_1+m_2+\cdots+ m_\ell}}{\partial t^{m_1+m_2+\cdots+ m_\ell}}
    \end{equation}    
where $m_1,m_2,\cdots,m_\ell$ are non negative integers, and $*$ denotes non specified coefficients.
In particular, the differential operator $S_{\ell}^{\delta}$ depends on $M_j^{\delta}$ and $\frac{\partial^j}{\partial t^j}$, with $1\le j\le \ell$ only, and the degrees of each $M_j^\delta$ is at most $\ell$.
\end{lemma}
\begin{proof}
The proof follows directly from the classical Fa\`a di Bruno formula, for higher derivatives of compositions applied to \eqref{eq.S_j^delta}, and the fact that $\frac{\partial^j}{\partial \epsilon^j}(P_\epsilon^\delta(t))_{|\epsilon=0}=j!M_j^\delta(t)$.

\end{proof}
\begin{corollary}\label{Prop.S_ell=M_ell}
   The operator $S_\ell^\delta$ applied to the identity gives the Melnikov function $M^\delta_\ell$. That is,   $S_\ell^\delta(id)=M^\delta_\ell$. 
\end{corollary}

\subsection{Matrix representation}\label{subsect.Matrix-repr}

 Note that an element $f_0+\epsilon f_1+\epsilon^2f_2+\cdots+\epsilon^kf_k\in V_k$, can be written as the vector $(f_0,f_1,f_2,\dots,f_k)$, where $\epsilon^j$ corresponds to the basic vector $e_j=(0,\dots,1,\dots,0)$.
Then, $\p^\delta_\epsilon:V_k\rightarrow V_k$ has a matrix representation $T_{\delta,k}=(T_{ij})$, with $T_{ij}:\Hol\rightarrow\Hol$ differential operators, where it acts on $(f_0,f_1,f_2,\dots,f_k)$ by evaluation on the $T_{ij}$.
\begin{lemma}\label{lem.Toeplitz}
 The matrix $T_{\delta,k}:\Hol^{k+1}\rightarrow\Hol^{k+1}$ is a lower triangular Toeplitz matrix. That is, it is of the form
\begin{equation}
    T_{\delta,k}=\begin{pmatrix}
    T_{00}&0&0&0&\cdots\\
    T_{10}&T_{00}&0&0&\cdots\\
    T_{20}&T_{10}&T_{00}&0&\cdots\\
    \vdots&\vdots&\ddots&\ddots&\vdots \\
    T_{k0}&T_{k-1,0}&\cdots&T_{10}&T_{00}
\end{pmatrix}.
\end{equation}

\end{lemma}
\begin{proof}

Note that, $T_{ij}(0)=0$, since $T_{ij}$ are differential operators and for any $f\in \Hol$,  the first column evaluated in $f$ is given by applying the matrix $T_{\delta,k}$ to the vertical vector $(f,0,\dots,0)^\top$. Here, and in the sequel, we denote by $\top$ the transposition. In general, the $j$-th column evaluated in $f$ is given by applying the matrix $T_{\delta,k}$ to the vector $(0,\dots,f,\dots,0)^\top$, with $f$ on the $j$-th place.
Observe that,  $v_j:=(0,\dots,f,\dots,0)^\top$ corresponds to $\epsilon^jf$, thus $T_{\delta,k}v_j$ corresponds to $\p^\delta_{\epsilon,k}(\epsilon^jf)$.
On the other hand, by the definition of the operator $\p^\delta_{\epsilon,k} $ it follows that, for any $j\ge0$, $\p^\delta_{\epsilon,k}(\epsilon^j f)=\epsilon^j \p^\delta_\epsilon(f)$. This implies that in the image of $T_{\delta,k}v_j$ the first $j-1$ terms are zero. Moreover, $\p^\delta_{\epsilon}(\epsilon^j f)=\epsilon \p^\delta_\epsilon(\epsilon^{j-1}f)$, which means that starting from the $j$-th term, all terms of  $T_{\delta,k}v_{j}$  are obtained from the previous column $T_{\delta,k}v_{j-1}$, by a shift by one downwards.
In particular, the matrix $T_{\delta,k}$ is lower triangular.

\end{proof}
\begin{lemma}\label{lemma:Matrix-T_{delta,k}}
 In the matrix $T_{\delta,k}$, the operators are given by $T_{ij}=S^\delta_{i}$. That is,
    
\begin{equation}\label{eq.matrix-T_delta}
    T_{\delta,k}=\begin{pmatrix}
    S^\delta_0&0&\cdots&0&0\\
    S^\delta_1&S^\delta_0&0&\cdots&0\\
    S^\delta_2&S^\delta_1&S^\delta_0&0&0\\
    \vdots&\vdots&\ddots&\ddots&\vdots \\
    S^\delta_{k}&S^\delta_{k-1}&\cdots&S^\delta_1&S^\delta_{0}
\end{pmatrix}.\end{equation}
\end{lemma}
\begin{proof}
    This follows from equation \eqref{eq.Op_P_epsilon}, then the first column is given by the vector 
    $$(S^\delta_0(f), S_1^\delta(f),\cdots,S^\delta_{k}(f))^\top.$$ Now the result follows from Lemma \ref{lem.Toeplitz}.  
\end{proof}

Moreover, this matrix representation has a good behavior with respect to the holonomy of products of loops, as shown in the following Proposition.

\begin{proposition}
  The matrix representation of the operator $\p_\epsilon^{\alpha_1\alpha_2}$ and of $(\p_\epsilon^{\delta})^{-1}$ corresponds to $T_{\alpha_1}T_{\alpha_2}$ and $(T_{\delta})^{-1}$, respectively. In particular, the matrix representation of $\p_\epsilon^{[\alpha_1,\alpha_2]}$ is $[T_{\alpha_1},T_{\alpha_2}]$, for any loops $\alpha_1\,,\alpha_2$ and $\delta$.
\end{proposition}
\begin{proof}
We know, by expression \eqref{eq.Op_P_epsilon}, that  $$\p^{\alpha_1\alpha_2}_{\epsilon}(f)=S_0^{\alpha_1\alpha_2}(f)+\epsilon S_1^{\alpha_1\alpha_2}(f)+\epsilon^2 S_2^{\alpha_1\alpha_2}(f)+\cdots+\epsilon^k S_k^{\alpha_1\alpha_2}(f).$$
  On the other hand, by definition,
  $\p^{\alpha_1\alpha_2}_{\epsilon}(f)=f\circ P^{\alpha_1\alpha_2}_\epsilon$ and $P^{\alpha_1\alpha_2}_\epsilon=P^{\alpha_2}_\epsilon\circ P^{\alpha_1}_\epsilon$, so $$\p^{\alpha_1\alpha_2}_{\epsilon}(f)=f\circ (P^{\alpha_2}_\epsilon\circ P^{\alpha_1}_\epsilon)=(f\circ P^{\alpha_2}_\epsilon) \circ P^{\alpha_1}_\epsilon=\p^{\alpha_1}_\epsilon(\p^{\alpha_2}_\epsilon(f)).$$
  Therefore, $\p^{\alpha_1\alpha_2}_{\epsilon}=\p^{\alpha_1}_\epsilon\circ \p^{\alpha_2}_\epsilon$.
  Hence, using the correspondence between the operators $\p^\delta_{\epsilon}$ and their matrix representations $T_\delta$, we have
  $T_{\alpha_1\alpha_2}(v_1)=T_{\alpha_1}(T_{\alpha_2}(v_1))$, for $v_0=(f,0,\dots,0)^\top$. For $v_j$ given by $(0,\dots,f,\dots,0)$, with $f$ in the $j-$th place, note that $T_{\alpha_1\alpha_2}(v_j)$, corresponds to  $$\p^{\alpha_1\alpha_2}_{\epsilon}(\epsilon^jf)=\epsilon^j\p^{\alpha_1\alpha_2}_{\epsilon}(f)=\epsilon^j\p^{\alpha_1}_\epsilon(\p^{\alpha_2}_\epsilon(f))=\p^{\alpha_1}_\epsilon(\p^{\alpha_2}_\epsilon(\epsilon^j f)).$$
  This last expression corresponds to $T_{\alpha_1}(T_{\alpha_2}(v_j))$. As $v_0,\dots,v_k$ form a basis for $V_k$, then $T_{\alpha_1\alpha_2}=T_{\alpha_1}T_{\alpha_2}$, thus showing that the matrix representation of $\p_\epsilon^{\alpha_1\alpha_2}$ is $T_{\alpha_1}T_{\alpha_2}$. The second and third relation, now follow directly.
  
\end{proof}

The relation, given in Proposition \ref{Prop.S_ell=M_ell}, between operators and Melnikov functions imposes a condition between loops and operators. More precisely, as is known, if $\sigma_j$ is a commutator of length $j$, i.e., $\sigma_j\in L_j$, then $M_\ell^{\sigma_j}$ vanishes for $\ell<j$, because $M_\ell$ is an iterated integral of length at most $\ell$. Therefore, $S_\ell^{\sigma_j}$ also vanishes, if $\ell<j$. This motivates the following definition.

\begin{definition}We say that a lower triangular Toeplitz matrix $T$;
$$T=\begin{pmatrix}
    A_0&0&0&\cdots\\
    A_1&A_0&0&\cdots\\
    A_2&A_1&A_0&\cdots\\
    \vdots&\vdots&\vdots&\ddots 
\end{pmatrix}$$
is \emph{nilpotent of coindex} $k$, if $A_\ell=0$, for $0\le \ell\le k-1$.
   
\end{definition}
\begin{lemma}\hfill
\begin{enumerate}
    \item[(i)]  Let $N_1$ be nilpotent of coindex $k$ and $N_2$ be nilpotent of coindex $s$, then the product $N_1N_2$ is nilpotent of coindex $(k+s)$.
   
   \item[(ii)] Let $N$ be nilpotent of coindex $k$, and let $T:=Id+N$, then $T$ is invertible and $T^{-1}-Id$ is also nilpotent of coindex $k$.
\end{enumerate} 
\end{lemma}
\begin{proof}
    This is straightforward from calculations.
\end{proof}

\begin{lemma}
If $\sigma_k\in L_k$, then   $T_{\sigma_k}=Id+N_{\sigma_k}$, where $N_{\sigma_k}$ is nilpotent of coindex $k$.
\end{lemma}

\begin{definition}
    Let $T_\delta=Id+N_\delta$ be the matrix, where $\delta\in \pi_1(F^{-1}(t_0),p_0)$, and $\pi_1(F^{-1}(t_0),p_0)$ is freely generated by $\{\delta_i\}_{i=1}^m$. 
    \begin{enumerate}
        \item[(i)] 
    We say that $N_\delta$ is $r$-\emph{order improved}
  if $S_\ell^\delta$ in \eqref{eq.matrix-T_delta} depends only on $S_j^{\delta_i}$, for $1\le j\le \ell-r+1$ and $1\le i\le m$.

\item[(ii)]
For a nilpotent matrix $N$ which is nilpotent of coindex $k$ and $k$-order improved, we say that it is \emph{$k$-triangular}. 
    \end{enumerate}

\end{definition}
\begin{example}
    Let $\sigma_2=[\delta_1,\delta_2]$, then $P^{\sigma_2}_\epsilon=Id+\epsilon^2S_2^{\sigma_2}+\cdots$,
    where $S_2^{\sigma_2}=S_1^{\delta_1}S_1^{\delta_2}-S_1^{\delta_2}S_1^{\delta_1}$.
\end{example}
\begin{lemma}\label{lemma:good}
    Let $N_{\alpha_1}$ be $k$-triangular and $T_{\alpha_2}$ be $s$-triangular, then ${N}_{\alpha_1}{N}_{\alpha_2}$ is $(k+s)$-triangular. In particular, ${N}_{\sigma_k}$, is $k$-triangular,  for $\sigma_k\in L_k$.
\end{lemma}
\begin{proof}

	Since ${N}_{\alpha_1}$ is $k$-order improved and ${N}_{\alpha_2}$ is $s$-order improved, then, by definition, $S_{\ell}^{\alpha_1}=0$, for $\ell=1,\dots,k-1$, and $S_{\ell}^{\alpha_2}=0$, for $\ell=1,\dots,s-1$.

Let $N_{\alpha_1}N_{\alpha_2}=(\hat{S}_\ell^{\alpha_1\alpha_2})_\ell$, then by direct multiplication,

$$
\hat{S}_\ell^{\alpha_1\alpha_2}=\sum_{r+h=\ell} S_{r}^{\alpha_1} S_{h}^{\alpha_2},
$$
with $k\le r$ and $s\le h$. Since $S_{\ell}^{\alpha_1}$ depends only on $S_j^{\alpha_i}$, with $1\le j\le \ell-k+1 $ and $S_{\ell}^{\alpha_2}$ depends only on $S_j^{\alpha_i}$, with $1\le j\le \ell-s+1 $, then $\hat{S}_\ell^{\alpha_1\alpha_2}$  depends only on $S_j^{\alpha_i}$, with $1\le j\le \ell-(k+s)+1$.

  In particular, denoting $N_1=N_{\alpha_1}$ and $N_2=N_{\alpha_2}$, and  writing $(Id+N_1)^{-1}=Id+N_1'$ and $(Id+N_2)^{-1}=Id+N_2'$,  it follows that $N_1+N_1'+N_1N_1'=N_2+N_2'+N_2N_2'=0$. The commutator $[T_1,T_2]=[Id+N_1,Id+N_2]$ is given by
    \begin{align*}
        [Id+N_1,Id+N_2]=&(Id+N_1)(Id+N_2)(Id+N_1')(Id+N_2')\\
        =&Id+N_1N_2+N_2N_1'+N_1N_2N_1'+N_1N_2N_2'+&\\
        +&N_2N_1'N_2'+N_1N_2N_1'N_2'.
    \end{align*}

    From Lemma \ref{lemma:good}, it follows that the products $N_1$, $N_1'$, $N_2$, $N_2'$ are $(k+s)$-triangular. 
    In particular, for $\sigma_k$ in $\OO_k$, we have a product of $k$ matrices which are 1-triangular. This gives a $k$-triangular matrix. 

\end{proof}

\begin{proof}[Proof of Theorem \ref{Theo.M-sigma_k}]
 By Lemma \ref{lemma:good}, we have that $S_\ell^{\sigma_k}$ depends only on $S_j^{\delta_i}$, with $1\le i\le m$ and $1\le j\le \ell-(k-1) $. On the other hand, by Corollary \ref{Prop.S_ell=M_ell}, it follows that $S_\ell^{\sigma_k}(\id)=M_\ell^{\sigma_k}$ and $S_j^{\delta_i}(\id)=M_j^{\delta_i}$, thus completing the proof.

\end{proof}

\subsection{Relation with the first proof} \label{sub7.3}
 Let $f\in\mathfrak{g}$. We define $F:=\id+f=\id+\epsilon f_1+\epsilon^2f_2+\cdots$ an element in $\mathcal{G}$. Then we have a differential operator $\mathbb{F}=S_0+\epsilon S_1+\epsilon^2S_2+\cdots$, as it was defined in Section \ref{sect.Diff.Op.Hol}. Then, $\mathbb{F}$ is associated to a Toeplitz matrix $T_F$, given by
    $$T_F=\begin{pmatrix}
    S_0&0&0&\cdots\\
    S_1&S_0&0&\cdots\\
    S_2&S_1&S_0&\cdots\\
    \vdots&\vdots&\vdots&\ddots 
\end{pmatrix},$$
where $S_\ell=\frac{\partial^\ell F}{\partial\epsilon^\ell}_{|\epsilon=0}$, which, by Fa\`a di Bruno (Lemma \ref{Faa}), depends on $f_1,\dots,f_\ell$. This association defines an isomorphism between $\mathfrak{g}$ and the space of lower triangular Toeplitz matrices $\mathcal{T}$, that is 
\begin{equation}\label{eq:iso}  
\Xi:\mathfrak{g}\rightarrow \mathcal{T}, \, \, f\mapsto T_F.
\end{equation}

\begin{proposition}\label{prop:bridge}
The triangularity index of $f$ is $\tau(f)=k$, if and only if, the matrix $T_F$ verifies $T_F-Id$ is $k$-triangular.    
\end{proposition}
\begin{proof}
   
Assume that $T_F$ is $k$-triangular. It means that $$T_F=\begin{pmatrix}
    S_0&0&0&0&\cdots\\
    0&S_0&0&0&\cdots\\
    \vdots&\ddots&\ddots&\vdots&\cdots\\
    S_k&0&0&S_0&\cdots\\
    \vdots&\ddots&\ddots&\ddots&\ddots 
\end{pmatrix},$$ and, for $k\le\ell$, $S_\ell$ depends on $S_j^{\delta_i}$, with $1\le j\le \ell-k+1$ and $1\le i\le m $, only.
On the other hand, $\lambda(f)=\min_r\{f_\ell\in \mathcal{A}_r\}$. That is, $\lambda(f)=\ell-k+1$.
Thus,
$$\tau(f)=\inf_\ell\{\ell-\lambda(f_\ell)+1\}=\inf_\ell\{\ell-\ell+k-1+1\}=k.$$
Now we prove the other direction. Assume that $\tau(f)=k$, that is $k=\inf_\ell\{\ell-\lambda(f_\ell)+1\}$. Thus, $k\le \ell-\lambda(f_\ell)+1$, therefore $1\leq\lambda(f_\ell)\le \ell-k+1$, for all $\ell$, and $k\le\ell$, which means that $T_F-Id$ is $k$-triangular.
\end{proof}

\section{Examples}\label{Sec:Examples}

In this Section we compute the universal Noetherianity index  $n_{\scriptscriptstyle H,\gamma}(k)$ given in \eqref{eq:n} for Examples \ref{exm:1} to \ref{exm:square}. Moreover, for $k=1$, we show that the Noetherianity index $\nu_{\scriptscriptstyle H,\gamma}(k)$, given in Definition \ref{def:noether} coincides with  $n_{\scriptscriptstyle H,\gamma}(k)$.
To achieve it,  we compute the Noetherianity index $n_r$ defined in Definition \ref{def:n_r}, for any $r\in\N$, thus proving that $n_r=\nu_r$, for $r=1$, (see \eqref{eq:nu_r}).

\begin{example}\label{exm:1}
     Let $H\in\C[x,y]$ be a generic polynomial and $\gamma$ a loop, non-trivial in homology and let  $\eta$ be a polynomial 1-form. We consider the deformation \eqref{eq:deformation}. We claim that, for a given $k$, the number $n_{\scriptscriptstyle H,\gamma}(k)$ is equal to $k$.\\
     
     We first show that $n_1=1$. 
    Indeed, if $M_1^\gamma=\int_\gamma\eta\equiv 0$, then, since the orbit under monodromy of $\gamma$ generates the whole homology group, we have that there exist polynomials $g$ and $R$,  such that $\eta=gdH+dR$ \cite{Ily69}. Then, by Fran\c coise algorithm \cite{F}, it follows that $M_2^\gamma=\int_\gamma g\eta$, which is an Abelian integral. 
    
    Moreover, by the Fran\c coise-Pelletier algorithm  \cite{FP}, $M_3^\gamma$ is given by $\int_\gamma (\eta'\eta)'\eta$, where $\eta'$ is the Gelfand-Leray derivative of $\eta$, with respect to $dH$, plus iterated integrals of smaller length. Since $\eta'=dg$, then $\int_\gamma (\eta'\eta)'\eta=\int_\gamma (dg\eta)'\eta$, which is of length at most 2. Therefore $M_3^{\sigma_3}$ vanishes. The same argument holds for any $M_j^{\sigma_j}$, $j>3$, thus showing that $n_1=1$. As $\nu_1\leq n_1=1$, it follows that $\nu_1=1$, too.

    Now, we show that $n_2=2$. To do this, suppose that $M_2^{\gamma}\equiv 0$. On the other hand, from the condition above, we know that $M_2^{\gamma}(t)=\int_{\gamma}g\eta$. Denote $\eta_1=g\eta$, then $M_2^\gamma=\int_\gamma\eta_1$. Repeating the above argument, $\eta_1=g_1dH+dR_1$, so $M_3^\gamma=\int_\gamma g_1\eta_1$. Hence, $M_3^\gamma$ is an Abelian integral. Again, by Fran\c coise-Pelletier, $M_4^\gamma$ is given by $\int_\gamma \eta_1'\eta_1$, plus an Abelian integral depending on $\eta_1$. Therefore, $M_4^\gamma$ is an iterated integral of length at most 2. 
    
    The same argument shows that, for any $j$, $M_j^\gamma$ is an iterated integral of length at most $j-2$, so $M_j^{\sigma_j}$ vanishes. Thus, $n_2=2$.
    Moreover, $2\leq\nu_2\leq n_2=2$, gives $\nu_2=2$, too.

     Analogously, it follows that, for any $r$, $n_r=\nu_r=r$. Hence, $$n=\max\{n_r\colon 1\le r\le k\}=\max\{r\colon 1\le r\le k\}=k,$$
     which is the universal Noetherianity index $n_{\scriptscriptstyle H,\gamma}(k)$ given in \eqref{eq:n}, coinciding with the Noetherianity index $\nu_{\scriptscriptstyle H,\gamma}(k)$.

     Note that, in this case the orbit generates the whole homology of the regular fiber, and hence the orbit length of iterated integrals coincides with the usual length.    
    \end{example}

\begin{example}[Hamiltonian triangle]\label{exm.Triangle}
 Let $H\in\C[x,y]$ be a non generic polynomial given by the product of three lines in general position. Without loss of generality, we assume that $H$ is given by $H=xy(x+y-1)$ and let $\gamma$ be the vanishing cycle at the center singular point.
 
 Let $\eta$ be a polynomial 1-form. We consider the deformation \eqref{eq:deformation}. 
  \begin{figure}[h]
         \centering
         \includegraphics[width=0.35\linewidth]{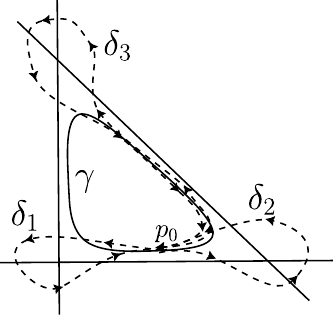}
         \caption{Hamiltonian triangle.}
         \label{fig:placeholder}
     \end{figure}
 
 We first show that $n_1=2$. 
 
 Note that the orbit $\OO$, under monodromy of $\gamma$,  is given by $\OO=\OO_1\oplus\OO_2\oplus\OO_3\oplus\cdots$. Here the vector spaces are: $\OO_1=(\gamma, \delta_1\delta_2\delta_3)$, $\OO_j=L_j$, with $j\ge 2$, where $\delta_1, \delta_2, \delta_3$ are the vanishing cycles at the critical points $(0,0), (1,0), (0,1)$, respectively, and with intersection number $<\gamma,\delta_j>=1$, for each $j=1,2,3$.
 We study  first the condition $M_1^\gamma=0$, that is $\int_\gamma\eta=0$. By \cite{U1}, it implies  
 \begin{equation}\label{eq:eta-triangle}
     \eta=gdH+dR+p_1(H)\frac{dx}{x}+p_2(H)\frac{dy}{y},
 \end{equation}
 where $g,R\in\C[x,y]$ and $p_1,p_2\in\C[H]$.
 
 To compute $M_2^{\sigma_2}$, for $\sigma_2=[\delta_1,\delta_2]$, we observe that it is defined by $M_2^{\sigma_2}(t)=W(\int_{\delta_1}\eta,\int_{\delta_2}\eta)$, where $W$ denotes the Wronskian \cite{MNOP2}.  Note that,
 $\int_{\delta_1}\eta=2\pi i(p_1-p_2)$, $\int_{\delta_2}\eta=2\pi ip_2$. Then 
 
 \begin{equation}\label{eq:M2}
     \begin{aligned}
    W\left(\int_{\delta_1}\eta,\int_{\delta_2}\eta\right)=&W\left(2\pi i(p_1-p_2), 2\pi i p_2\right)\\
    =&(2\pi i)^2\left((p_1-p_2)p_2'-(p_1-p_2)'p_2\right)\\
    =&-4\pi^2W(p_1,p_2).
 \end{aligned}
 \end{equation}
This is identically zero, if and only if,  $p_1$ and $p_2$ are proportional.

Then, $n_1$ cannot be equal to 1, as $p_1$ and $p_2$ can be chosen arbitrarily, in particular they can be non-proportional, and therefore $M_2^{\sigma_2}\not\equiv0$. Hence, a second condition requiring the vanishing of \eqref{eq:M2} is necessary. 

Since $M_j^{\sigma_j}$ is also given in terms of the Wronskians of $p_1$ and $p_2$, then it follows that $M_j^{\sigma_j}\equiv0$, for all $j\ge 3$. Therefore, $n_1=2$. 
Moreover, by \eqref{eq:eta-triangle}, any couple of polynomials $(p_1,p_2)$ can be realized by a convenient form $\eta$. Hence $\nu_1=n_1=2$.

We now verify that $n_2$ is equal to $3$.
Note that $M_2^{\sigma_2}\equiv 0$, for all $\sigma_2\in\OO_2=L_2$, means that $M_2^{\gamma}$ is an Abelian integral. That is, there exists a rational $1$-form $\tilde\eta$, such that, $M_2^{\gamma}=\int_{\gamma}\eta'\eta=\int_\gamma\tilde\eta$. More explicitly, it can be verified that 
\begin{equation}\label{eq:eta-tilde-1}
    \tilde\eta=gdR-p_1(t)gd\varphi+p'_1(t)Rd\varphi,
\end{equation}
where $\varphi=\log(xy^c)$, with $c\in\C$. Then, the second diagonal correspond to the deformation $dH+\epsilon\tilde\eta=0$, it means that in the second diagonal we have the same differential algebra $\A_1$. Hence, by expression \eqref{eq:D_r-n_r} the universal Noetherianity index $n_2$ is equal to $3$ (the ascending chain of radical ideals stabilizes after the second condition).

Repeating the same argument, we can conclude that $n_r=r+1$, for any $r\in\N$. Hence, for any $k\in\N$,
$$n_{\scriptscriptstyle H,\gamma}(k)=\max\{n_r\colon 1\le r\le k\}=\max\{r+1\colon 1\le r\le k\}=k+1.$$ 
\end{example}
\begin{remark}
In this triangle case, we conjecture that $\nu_r=n_r$, for $r\ge 2$, and hence that the optimal bound is $\nu_{H,\gamma}(k)=n_{H,\gamma}(k)=k+1$, for any $k\in\N$.

In order to prove that $\nu_2=n_2$ one has to construct a form $\eta$ such that $\ev_\eta(C_1)=0$, with the additional condition $\ev_\eta(m_{21}(\gamma))=0$, but $\ev_\eta(m_{32}(\sigma_2))\ne0$. One of the main difficulties is that already on the level of the second diagonal the cohomology class of $\eta$ is not sufficient to express the Melnikov functions. Even if we work with abelian integrals on the second diagonal, the Melnikov function $\ev_{\tilde\eta}(m_{21}(\gamma))=\int_{\gamma}\tilde\eta$, depends on the $1$-form $\tilde\eta$, which is not free, but depends on the form $\eta$. Moreover, it does not depend only on the coholomology class of $\eta$, but also on the relative exact part of $\eta$. It is not clear that the left-hand side of the relation $\ev_{\tilde\eta}(m_{32}(\sigma_2))\ne0$ is not always zero for the forms $\tilde\eta$ induced by $\eta$. Of course, passing to further diagonals in general we have to work with the full homotopy and not homology and the calculations become more and more complicated. 

In the following examples, we also conjecture that the Noetherianity index $\nu_{H,\gamma}(k)$ giving the optimal bound coincides with the universal Noetherinity index   $n_{H,\gamma}(k)$, but the same difficulties in proving the conjecture apply as in the triangle case.  
    
\end{remark}

\begin{example}[Product of lines in general position]\label{exm:product-lines}
Let $H=f_1\cdots f_d$ be a product of polynomials $f_j$, with $\deg (f_j)=1$ and let $\ell_j=\{f_j=0\}$ be  lines in general position.
We assume, moreover, that each critical value of $H$, except $0$, corresponds to only one critical point (which is of center type).
Let $\gamma_i$ be the cycles vanishing at a center critical points, and let $\delta_i$ be the cycles vanishing at saddle points, we choose orientations so that the intersection index $(\gamma_i,\delta_j)$ is 1, if they intersect, or 0 otherwise \cite{Acampo}. 
We label the cycles $\delta_i$, so that $\delta_i$, $i=1,\dots, d-1$, are the cycles vanishing at the intersection point between the lines $\ell_i$ and $\ell_d$. Moreover, we suppose that the lines $\ell_i$ are ordered by the order of their intersection points with $\ell_d$. We consider a deformation $dH+\epsilon\eta=0$ and let $\gamma$ be a cycle $\gamma_i$. 
\begin{figure}[h]
    \centering
    \includegraphics[width=0.7\linewidth]{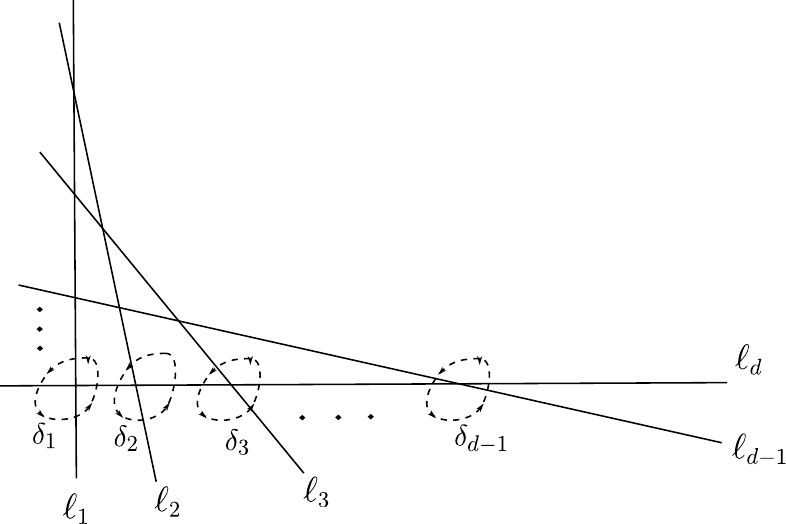}
    \caption{Lines in general position.}
    \label{fig:placeholder}
\end{figure}

Now we compute $n_1$ and $\nu_1$. First condition, $M_1^\gamma\equiv 0$ gives $\int_\gamma\eta\equiv0$. By \cite{MU}, it is known that $L_2=\OO_2$ and that
\begin{equation}\label{eq:U}
\eta=gdH+dR+p_1(H)\frac{df_1}{f_1}+\dots+p_{d-1}(H)\frac{df_{d-1}}{f_{d-1}}.
\end{equation}
Let $\sigma_2=[\delta_i,\delta_j]\in\OO_2,$ with $1\le i ,j\le d-1$ and $i-j$ odd. Then, second condition $M_2^{\sigma_2}=0$ gives 

$$\begin{aligned}
    M_2^{\sigma_2}&=W(\int_{\delta_i}\eta,\int_{\delta_j}\eta)\\
    &=W\left(\int_{\delta_i}p_1(H)\frac{df_1}{f_1}+\dots+p_{d-1}(H)\frac{df_{d-1}}{f_{d-1}},\int_{\delta_j}p_1(H)\frac{df_1}{f_1}+\dots+p_{d-1}(H)\frac{df_{d-1}}{f_{d-1}}\right)\\
    &=\pm W( 2\pi ip_i(t), 2\pi ip_j(t))\\
    &=\pm(2\pi i)^2W(p_i(t),p_j(t)),
\end{aligned}$$
    which vanishes if and only if $p_i$ and $p_j$ are proportional, giving $n_1\geq 2$. On the other hand, note that $M_j^{\sigma_j}$ are given by nested Wronskians, including $W(p_i(t),p_j(t))$, which vanishes. Hence, $M_j^{\sigma_j}\equiv0$, for all $j\ge3$, showing that $n_1=2$. Moreover, from \eqref{eq:U},
    any set of polynomials $p_i$, can be realized by a convenient form $\eta$, thus giving that $\nu_1=2$, too.
    
    Now we show that $n_r=r+1$, for all $r\ge2$, analogously to the Triangle case, Example \ref{exm.Triangle}.  
    
    Hence, by the definition \eqref{eq:n} we have that,
    $$n_{\scriptscriptstyle H,\gamma}(k)=\max\{r+1\colon 1\le r\le k\}=\max\{r+1\colon 1\le r\le k\}=k+1.$$ 
\end{example}

\begin{example}\label{exm:square}

Let $H=(x^2-1)(y^2-1)$, $f_1=x+1, f_2=x-1, f_3=y+1$, $f_4=y-1$, $\eta_i=d\log f_i$. Clearly $\eta_i'=0$ (Gelfand-Leray derivatives). 
Let $\gamma$ denote the real loop of $\{F=t\}$ vanishing at the origin and $
\delta_i$, $i=1,\ldots,4$ the loops vanishing at the saddles $(\pm1, \pm1)$ starting from $(-1,-1)$ counterclockwise such that $\int_{\delta_1\delta_2}\eta_3=0$, $\int_{\delta_2\delta_3}\eta_2=0$, $\int_{\delta_1}\eta_1=2\pi i$. The fundamental group $\pi_1(\{H=t\},p_0)$ is freely generated by $\gamma$ and $\delta_i$, $i=1,\ldots,4$.
 
\begin{figure}[h]
         \centering
         \includegraphics[width=0.35\linewidth]{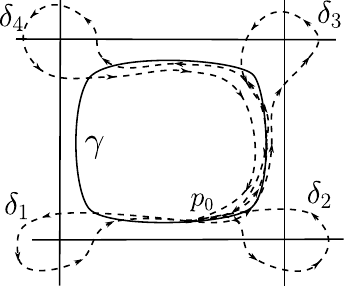}
         \caption{Hamiltonian Square.}
         \label{fig:placeholder}
     \end{figure}

We give orientation to the cycles such that $(\gamma,\delta_j)=1$ and
\begin{equation}\label{eq:basic integrals square nonzero}
\int_{\delta_1}\eta_1=-\int_{\delta_1}\eta_3=
\int_{\delta_2}\eta_3=-\int_{\delta_2}\eta_2=\int_{\delta_3}\eta_2=
2\pi i,
\end{equation}
while 
\begin{equation}\label{eq:basic integrals square zero}
   \int_{\delta_2}\eta_1 =\int_{\delta_3}\eta_1=\int_{\delta_1}\eta_2=\int_{\delta_3}\eta_3=0.
\end{equation}

Thus, the orbit $\OO\subset \pi_1$
is generated by 
\begin{equation}\label{eq:generators of orbit}
\gamma, \delta_1\delta_2\delta_3\delta_4\,, v_2=[\delta_1\delta_2,\delta_2\delta_3], ...,v_i=[\delta_1\delta_2,[\delta_2,[\dots,[\delta_2,\delta_2\delta_3]\dots],...
\end{equation}
and $[\OO, \pi_1]$.

 Fix a perturbation form $\eta$, and consider the perturbation $dH+\epsilon \eta=0$. We show that $n_1=3$.
 
First, note that the condition $M_1^{\gamma}=0$
implies that \begin{equation}\label{eq.eta-P_j}
    \eta=gdH+dR+p_1(H)\eta_1+p_2(H)\eta_2+p_3(H)\eta_3,
\end{equation}
with $g,R\in\C[x,y]$ and $p_1,p_2,p_3\in\C[H]$, because $\eta_1,\eta_2,\eta_3$ generate the subspace of $H_1(H=t)$ orthogonal to $\OO$
(see Section 3 of \cite{MNOP2}).

Now, in order to compute $M_2^{v_2}$, note that,
\begin{equation}\label{eq:etas-deltas}
    \int_{\delta_1\delta_2}\eta=2\pi i (p_1-p_2),\ \int_{\delta_2\delta_3}\eta=2\pi ip_3,\ \int_{\delta_2}\eta=2\pi i (p_3-p_2).
\end{equation}

From Fran\c coise algorithm, it is known that $M_2^{v_2}=\int_{v_2}\eta'\eta$, where $\eta'$ denotes the Gelfand--Leray derivative of $\eta$, $\eta'=\frac{d\eta}{dH}$. Thus, using equations \eqref{eq:etas-deltas}, the condition $M_2^{v_2}\equiv0$ means that
$$
\begin{aligned}
    M_2^{v_2}&=\int_{[\delta_1\delta_2,\delta_2\delta_3]}\eta'\eta\\
    &=W\left(\int_{\delta_1\delta_2}\eta,\int_{\delta_2\delta_3}\eta\right)\\
    &=(2\pi i )^2W(p_1-p_2,p_3).    
\end{aligned}$$
Therefore, $M_2^{v_2}\equiv 0$ means that there exists $\mu\in\C$, such that $p_3=\mu(p_1-p_2)$, or $p_1-p_2=\mu p_3$. Assume  the first case, $p_3=\mu(p_1-p_2)$. Now we compute $M_3^{v_3}$, which is given by 
\begin{equation}
    \begin{aligned}
     M_3^{v_3}&=\int_{[\delta_1\delta_2,[\delta_2,\delta_2\delta_3]]}(\eta'\eta)'\eta\\
     &=W\left(\int_{\delta_1\delta_2}\eta,W\left(\int_{\delta_2}\eta,\int_{\delta_2\delta_3}\eta\right)\right)\\
     &=(2\pi i)^3W(p_1-p_2,W(p_3-p_2,p_3))
    \end{aligned}.
\end{equation}
Substituting $p_3=\mu(p_1-p_2)$, we get

\begin{equation}
    \begin{aligned}
     M_3^{v_3}&=(2\pi i)^3W(p_1-p_2,W(\mu(p_1-p_2)-p_2,\mu(p_1-p_2)))\\
     &=\mu(2\pi i)^3W(p_1-p_2,W(\mu(p_1-p_2)-p_2,(p_1-p_2)))\\
     &=\mu(2\pi i)^3W(p_1-p_2,W(-p_2,(p_1-p_2)))\\
     &=-\mu(2\pi i)^3W(p_1-p_2,W(p_2,p_1)).\\
    \end{aligned}
\end{equation}
Hence, $M_3^{v_3}\equiv0$ gives the condition $W(p_1-p_2,W(p_2,p_1))\equiv 0$.

Now, we compute $M_4^{v_4}$. It is given by 
\begin{equation*}
    \begin{aligned}
     M_4^{v_4}&=\int_{[\delta_1\delta_2,[\delta_2,[\delta_2,\delta_2\delta_3]]]}((\eta'\eta)'\eta)'\eta\\
     &=W\left(\int_{\delta_1\delta_2}\eta,W\left(\int_{\delta_2}\eta,W\left(\int_{\delta_2}\eta,\int_{\delta_2\delta_3}\eta\right)\right)\right)\\
    &=(2\pi i)^4W(p_1-p_2,W(p_3-p_2,W(p_3-p_2,p_3)))\\
    &=(2\pi i)^4W(p_1-p_2,W(\mu(p_1-p_2)-p_2,W(\mu(p_1-p_2)-p_2,\mu(p_1-p_2))))).
    \end{aligned}
\end{equation*}
Note that the Wronskian $W(\mu(p_1-p_2)-p_2,W(\mu(p_1-p_2)-p_2,\mu(p_1-p_2))))$ is either zero, or a multiple of the Wronskian $W(p_1,p_2)$. Since, $W(p_1-p_2,W(p_2,p_1))\equiv 0$, then $M_4^{v_4}\equiv 0$. The same argument shows that, for any $j\ge5$, $M_j^{v_j}\equiv 0$ vanishes as well. Hence, $n_1=3$.
We have $\nu_1=3$, because by \eqref{eq.eta-P_j}, any set of polynomials $p_i$, $i=1,2,3$, can be realized by a convenient form $\eta$.

We now calculate $n_2$. We assume that all the elements on the first diagonal in \eqref{table-M_i} vanish and that $M_2^\gamma\equiv0$.
From the vanishing on the first diagonal we get relation \eqref{eq.eta-P_j}, with $p_3$ and $w(p_1,p_2)$ both proportional to $p_1-p_2$. We consider the case $p_3=\mu(p_1-p_2)$. Then, 

$$\eta=(g-\sum_{i=1}^3 p_i'(H)\log(f_i))dH+d(R+\sum_{i=1}^3p_i(H)\log(f_i)).$$
Thus,

$$\begin{aligned}
    M_2^\gamma=\int_\gamma \eta'\eta&=\int_\gamma (g-\sum_{i=1}^3 p_i'(t)\log(f_i))\eta\\
    &=\int_{\gamma}g\eta-\sum_i p_i'(t)\int_\gamma\log(f_i)\eta\\
    &=\int_{\gamma}g\eta-\sum_i p_i'(t)\int_\gamma\log(f_i)d(R+\sum_{j=1}^3p_j(t)\log(f_j))\\
    &=\int_{\gamma}g\eta-\sum_i p_i'(t)\int_\gamma\log(f_i)dR-\sum_i p_i'(t)\sum_{j=1}^3p_j(t)\int_{\gamma}\log(f_i)d\log(f_j)\\
    &=\int_{\gamma}g\eta-\sum_i p_i'(t)\int_\gamma\log(f_i)dR-\sum_{i<j}W(p_i,p_j)\int_{\gamma}\log(f_i)d\log(f_j).\\
\end{aligned}
    $$ 
    Note that $\int_{\gamma}g\eta-\sum_i p_i'(t)\int_\gamma\log(f_i)dR$ is an Abelian integral, since $\int_\gamma\log(f_i)dR=\int_\gamma R\frac{df_i}{f_i}$. On the other hand,
    $$\begin{aligned}
\sum_{i<j}W(p_i,p_j)\int_{\gamma}\log(f_i)d\log(f_j))=&W(p_1,p_2)\int_{\gamma}\log(f_1)d\log(f_2))+W(p_2,p_3)\int_{\gamma}\log(f_2)d\log(f_3)\\
&+W(p_1,p_3)\int_{\gamma}\log(f_1)d\log(f_3))\\
=&W(p_2,\mu(p_1-p_2))\int_{\gamma}\log(f_2)d\log(f_3))\\
&+W(p_1,\mu(p_1-p_2))\int_{\gamma}\log(f_1)d\log(f_3)\\
=&\mu(W(p_1,p_2))\int_\gamma(\log(f_2)d\log(f_3))+\log(f_1)d\log(f_3))\\
=&\mu(W(p_1,p_2))\int_\gamma(\log(f_2f_1)d\log(f_3))\\
=&\mu(W(p_1,p_2))\int_\gamma(\log(\frac{F}{f_3f_4})d\log(f_3))\\
=&-\mu(W(p_1,p_2))\int_\gamma(\log(f_3f_4)d\log(f_3))\\
=&-\mu(W(p_1,p_2))\int_\gamma(\log(y^2-1)d\log(y+1)).
\end{aligned}$$
    
Then by the Cauchy Theorem, since $\log(y^2-1)d\log(y+1)$ is holomorphic in a neighborhood of $\gamma$, we get that $\int_\gamma(\log(y^2-1)d\log(y+1))\equiv0$. Therefore, we conclude that $M_2^{\gamma}=\int_\gamma\tilde{\eta}$, where 
\begin{equation}\label{eq:etatilde}
\tilde{\eta}=g\eta+ R(x,y)\left(\sum_{j=1}^3 p_j'(H)\frac{df_j}{f_j}\right),
\end{equation}
so it is an abelian integral and its
vanishing implies that the form $\tilde \eta$ is also of the form \eqref{eq.eta-P_j}. 
Considering the deformation $dH+\epsilon\tilde\eta$, we are exactly in the same situation as in the study of $n_1$, but with a shift in the index. 
It hence follows that $n_2=4$ as above.

The same argument applies for any $n_r$, for $r\geq2$.
We conclude that $n_r=r+2$, for $r\geq1$, and hence 
\begin{equation*}
   n_{\scriptscriptstyle H,\gamma}(k)= k+2\,.
\end{equation*}

\end{example}

\begin{remark}\label{rem:orbit-length-vs-usual-length}
    We stress that in this case, the orbit length and the usual length do not coincide. For instance, once we have proved that the first diagonal vanishes, then $M_2^\gamma$ is an iterated integral of orbit length $1$. However, for the commutator $[\delta_1,\delta_2]$, which does not belong to the orbit of $\gamma$,  we have $M_2^{[\delta_1,\delta_2]}\neq 0$. Hence, $M_2^{[\delta_1,\delta_2]}$ is not an iterated integral of length 1.
    \end{remark}
     \begin{corollary}\hfill
        \begin{itemize}
    
        \item[(i)]   for $H$ generic, $n_{\scriptscriptstyle H,\gamma}(k)=k$,
        \item[(ii)]  for $H$ a product of lines in general position, $n_{\scriptscriptstyle H,\gamma}(k)=k+1$,
        \item[(iii)]  for $H$ a product of two pairs of parallel lines, $n_{\scriptscriptstyle H,\gamma}(k)=k+2$.
    \end{itemize}
 \end{corollary}
 \begin{remark}

    One observes that in examples  \ref{exm:1} to \ref{exm:square}, when the first integral is more degenerate, the number 
    $n_{\scriptscriptstyle H,\gamma}(k)$ increases. In the first two cases, one gets $n_{\scriptscriptstyle H,\gamma}(k)=k+\kappa-1$, where $\kappa$ is the depth introduced in \cite{MNOP}. However, in the third case of a pair of parallel lines, the depth $\kappa$ is infinity. 
   
\end{remark}

\section{The Lie algebra of the principal diagonal}\label{Sec:Lie-alg}

\subsection{Homomorphism between Lie algebras}

Here we show that   the Melnikov functions $M_k^\delta(t)$ of $\delta\in L_k$ define a universal  Lie algebra homomorphism from $gr\pi_1$ into a free Lie algebra $\mathfrak{a}\subset \A_1$, a much smaller object then $\A_1$ , with a similar property of universality.  

The key fact  is that, up to higher order terms, the Poincar\'e return map is given by the flow of a vector field defined by its leading term. More precisely,
\begin{equation*}
    P_\ep^{\delta_i}=t+\ep M_1^{\delta_i}(t)+O(\ep^2)= \tau^\ep\left(M_1^{\delta_i}(t)\partial_t\right)+O(\ep^2).
\end{equation*}
Here, $\tau^\ep$ is the time-$\ep$ flow of the vector field $v_i=M_1^{\delta_i}(t)\partial_t$. 
Similarly, for $\delta\in L_k$, define $v_{k,\eta}=M_k^{\delta}(t)\partial_t$, so
\begin{equation*}
    P_\ep^{\delta}=\tau^{\ep^k}\left(v_{k,\eta}\right)+O(\ep^{k+1}).
\end{equation*}
Note that this is not the case for terms appearing in lower diagonals in  \eqref{table-M_i}.

    Let $\mathfrak{X}(\C,0)$ be the Lie algebra of the germs of holomorphic vector fields on $(\C,0)$. Let $ gr\pi_1$ be the graded algebra associated to the filtration $\pi_1=L_1\supset L_2\supset\dots$. Recall that $gr \pi_1$ is a   Lie algebra, with the Lie bracket induced by the commutator in $\pi_1(X)$, see \cite{SerreBook}.

\begin{proposition}\label{prop:8.1}
     The mappings $v_{k,\eta}: L_k\to \mathfrak{X}(\C,0)$ induce a Lie algebras homomorphism  $v_\eta: gr \pi_1 \to \mathfrak{X}(\C,0)$.

     In particular, given $\delta\in L_k$ and $\delta'\in L_j$, then $$v_{k+k',\eta}[\delta,\delta']=[v_{k,\eta}(\delta),v_{k',\eta}(\delta')]=W(M_k^\delta,M_{k'}^{\delta'})\partial_t,$$
     where $W$ denotes the Wronskian of two functions. 
\end{proposition}
\begin{proof} This immediately follows from the fact that a commutator of flows of two vector fields is the flow of commutators of the vector fields, up to higher  order terms in $\epsilon$.
\end{proof}

A universal version of $v_\eta$ can be defined using the universal holonomy $\scp$ as follows. 
Let $Diff(\A_1)$ be the linear space of the differentiations of $\A_1$, i.e., the linear maps $\A_1\to\A_1$ satisfying the Leibnitz rule.
It is a Lie algebra under the standard bracket $[\alpha,\beta]=\alpha\beta-\beta\alpha$.
There is a monomorphism $\iota:\A_1\to Diff(\A_1)$ defined by  $\iota(a)=a\partial$, where  $\partial$ denotes the differentiation $b\to b'$ in $\A_1$.
Clearly, $\left[\iota(a),\iota(b)\right]=\iota(ab'-a'b)$. 
Denote by $\mathcal{X}_1\subset\iota(\A_1)\subset Diff(\A_1)$ the Lie algebra generated by the elements $\iota(m_1(\delta_i))$.

By Lemma~\ref{lem:Gk small}, for $\delta\in L_k$ the differential polynomial $m_k(\delta)\in \A_1$. One can check, similarly as above, that
$$
\iota\left(m_{k+k'}([\delta,\delta'])\right)=\left[\iota(m_k(\delta)), \iota(m_{k'}(\delta'))\right],\quad \text{for}\,\,\delta\in L_k, \,\delta'\in L_{k'},$$
so the mappings $\iota\circ m_k$ induce a  Lie algebra homomorphism  $\mathfrak{P}: gr \pi_1 \to \mathfrak{X}_1$. Evidently, 
$$
v_\eta=\mathfrak{e}_\eta\circ\mathfrak{P}, \quad \text{where} \,\,\mathfrak{e}_\eta\left(\iota(m_1(\delta_i))\right)=M_1^{\delta_i}(t)\partial_t.
$$

\begin{remark}
    The exponential mapping $\exp:\mathfrak{g}\to \mathcal{G}$, $\exp(a)=1+\sum \frac{a^k}{k!}$, relates the group commutator in $\mathcal{G}$ with the Lie bracket above by  $\exp([a,b])=[\exp(a),\exp(b)]+\ep^{k+l}\mathfrak{m}$ for $a\in\ep^k\mathfrak{m}, b\in\ep^l\mathfrak{m}$.
\end{remark}

\begin{remark}\label{rem:wronskian}
    The classical Chen's theorem claims that, for any $\gamma\in\pi_1(\{H=t\})$, there exists a tuple of forms $\omega_{i_1},\dots,\omega_{i_k}$, such that the iterated integral $\int_\gamma\omega_{i_1},\dots,\omega_{i_k}\neq0$.
    This is the analogous to the classical de Rham duality, but in homotopy (instead of homology) and with iterated differential forms, instead of simple differential forms.
    Note that $\int_{[\gamma_1,\gamma_2]}\omega_1\omega_2=\det\left(\int_{\gamma_i}\omega_j\right)$, \cite{MNOP2}.

    In our settings of deformations, there is a similar formula 
    \begin{equation}
        m_{k+k'}([\delta,\delta'])=W\left(m_k(\delta), m_{k'}(\delta')\right),
    \end{equation}
    i.e., the Lie bracket now comes from a differential algebra $\A_1$, and not from the universal enveloping algebra of $gr \pi_1$.
    
\end{remark}

\subsection{Kernel of the evaluation map $\ev_\eta$ on the first diagonal $\A_1$.}
 
It is know that there exists a nontrivial identity $T_4\ne0$ in the free Lie algebra of one dimensional vector fields,
such that its evaluation vanishes for any $5$-tuple of holomorphic vector fields   \cite{Ki,R86}. Indeed, given any holomorphic vector fields $X_1,\ldots,X_5$, then the vector field
\begin{equation}\label{eq:T4-1}
T_4=\sum_{\beta\in S_4}\sgn(\beta) [X_{\beta(1)},[X_{\beta_{(2)}},[X_{\beta_{(3)}},[X_{\beta_{(4)}},X_5]]],
\end{equation} 
vanishes identically. 

We can ask ourselves if there exists a nontrivial kernel of the evaluation maps $\ev_\eta:\A_1\to\A_1^\eta.$ We answer this question in the following proposition.

\begin{proposition}\label{prop:k=0}
 For any non-zero differential polynomial $p\in\A_1$ there exists a one-form $\eta$ such that $\ev_\eta(p)\not\equiv0$, i.e. $\cap_{\eta}Ker(\ev_\eta)=\{0\}.$
\end{proposition}

\begin{remark}  
We first explain why this is not in contradiction with the example provided by $T_4$.
Given any $5$-tuple of loops $\delta_1,\ldots,\delta_5$ in $\pi_1$, consider the loop given by the $T_4$ relation:
  $$T_4(\delta_1,\dots,\delta_5)=\prod_{\beta\in S_4} [\delta_{\beta(1)},[\delta_{\beta_{(2)}},[\delta_{\beta_{(3)}},[\delta_{\beta_{(4)}},\delta_5]]]^{\sgn(\beta)}\in L_5.$$
Note that $T_4(\delta_1,\dots,\delta_5)\ne0$ in $\pi_1$.
Given a form $\eta$, let $v_{i,\eta}$, $i=1,\ldots,5$, be the vector fields realizing the holonomy along the loops $\delta_i$.
Then, the vector field $T_4(v_{1,\eta}\ldots,v_{5,\eta})$, associated to $T_4(\delta_1,\dots,\delta_5)$ vanishes by the $T_4$ relation. Recall that $v_{j,\eta}=M_1^{\delta_j}(t)\partial_t$. Therefore, $T_4(M_1^{\delta_1},\ldots,M_1^{\delta_5})$ vanishes identically in $\A_1$.

This exhibits a nontrivial element lying in the kernel of any homomorphism $v_\eta:gr\pi_1\rightarrow\mathfrak{X}(\C,0)$.
 
 Let $p\in\A_1$ such that
  $$\iota(p)=\sum_{\beta\in S_4} \sgn(\beta)[\iota(m_1({\delta_{\beta(1)})}),[\iota(m_1({\delta_{\beta_{(2)}}})),[\iota(m_1({\delta_{\beta_{(3)}}})),[\iota(m_1({\delta_{\beta_{(4)}}})),\iota(m_1(\delta_5))]]].$$
  Then, $\ev_\eta(p)=v_\eta(T_4)=0$.

However, $\iota(p)$ is also zero in $Diff(\A_1)$ (and therefore also $p$ is zero in $\A_1$), since relation $T_4$ is trivially satisfied for all Lie algebras coming from differential algebras with one differentiation, i.e., with Lie bracket defined by $[a,b]=ab'-a'b$. Indeed, for such a Lie bracket, all monomials on the left-hand side will cancel out after opening the brackets. This is not the case for the Lie algebra of vector fields on the plane, where the Lie bracket definition uses two differentiations. In fact, \cite{PR} shows that the opposite is true: the vanishing of $T_4$  is verified in a semiprime Lie algebra only if it can be embedded into a differential algebra, with the Lie bracket defined as above. 
\end{remark}

Let  $k$ be the  order of $p$, and consider $p$ as a polynomial in $m_1(\delta_i)^{(j)}$, $i=1,\dots,m$, $j=0,\dots,k$: 
$$
p=P\left(m_1(\delta_i)^{(j)}\right), \quad 0\neq P\in\C[x],\, x\in \C^{m(k+1)}.
$$

The proof of the Proposition  follows from the following:
\begin{lemma}\label{lem:K=0 at t0}
For any $x=(x_{ij})\in \C^{m(k+1)}$ and any regular point $t_0\in\C$ of $H$, there exists a polynomial one-form $\eta$ 
such that the Abelian integrals $I_i(t)=\int_{\delta_i(t)}\eta$, $i=1,...,m$, verify 
\begin{equation}\label{eq:Lem9.2}
    I_i^{(j)}(t_0)=x_{ij},\quad i=1,\dots,m, \,j=0,\dots,k.
\end{equation}
\end{lemma}

\begin{proof}[Proof of Lemma \ref{lem:K=0 at t0}] 
       Let $t_0$ be an arbitrary regular value of $H$.
           For $k=0$, this means that there exists a polynomial form $\eta_0$ such that $\int_{\delta_i(t_0)}\eta_0=x_{i0}$, for any $m$-tuple $x_{i0}$, $i=1,\cdots,m$. This follows from the algebraic de Rham theorem. Indeed, there exists a basis of polynomial one forms dual to the cycles $\delta_i$. Using it and resolving the corresponding system of equations, one shows the existence of the form $\eta_0$.

       Assume, now that $x_{ij}=\int_{\delta_i(t_0)}\eta_r$, for $j=0,...,r$. We construct a form $\eta_{r+1}$ verifying the condition on the level $r+1$, without modifying the previously achieved conditions. We put $\eta_{r+1}=\eta_r+(H-t_0)^{r+1}\nu_{r+1}$, where 
$\nu_{r+1}$ is the form such that 
           $$\frac{1}{(r+1)!}\int_{\delta_i(t_0)}\nu_{r+1}=x_{i,r+1}-\left(\int_{\delta_i(t_0)}\eta_{r}\right)^{(r+1)}.$$
We conclude by induction.
      \end{proof}     

\begin{proof}[Proof of the Proposition~\ref{prop:k=0}]
    Choose $x\in \C^{m(k+1)}$ such that $P(x)\neq0$, and let $\eta$ be the polynomial one-form as in Lemma~\ref{lem:K=0 at t0}.
    Then $\ev_\eta(p)(t_0)=p(I_1,\dots,I_m)=P(x)\neq0$.
\end{proof}

\subsection{Sharpness of the  upper bound $\nu_{H,\gamma}(1)\le n_{H,\gamma}(1)$}
In this section we consider forms $\omega$ holomorphic in $F^{-1}(U)$, where $U\subset\C$ is a domain. In other words, $\omega\in F_*\Omega^1(U)$. For such forms a literal analogue of Theorem~\ref{thm:main1}  holds, thus allowing us to  define a \emph{local Noetherianity index} $\nu^U_{H,\gamma}(k)$,  as in  Definition~\ref{def:noether}. Necessarily, $\nu_{H,\gamma}(k)\le \nu^U_{H,\gamma}(k)\le n_{H,\gamma}(k)$.
\begin{proposition}\label{prop:sharpness}
    There is some open set $U\subset\C$ and a form $\eta\in F_*\Omega^1(U)$  such that $\nu^U_{H,\gamma}(1)=n_{H,\gamma}(1)$. 
\end{proposition}

The proof is based on Ritt's theorem on zeros (see \cite[Theorem~5]{PLect}):

\begin{theorem}[Ritt's theorem on zeros]\label{thm:ritt of zeros}
Let $p_1, . . . , p_n\in\A_1$ be
differential polynomials. Then the system $p_1(m) = \dots = p_n(m) = 0$, where $m=(m_1(\delta_1),\dots,m_1(\delta_m))$,
has a power series solution $\mu=(\mu_1(t),\dots,\mu_m(t))$ with a nonzero convergence radius if and only if
$1\notin\langle p_1,\dots, p_n\rangle.$
\end{theorem}
\begin{proof}[Proof of Proposition \ref{prop:sharpness}]

In a standard way, Ritt's theorem on zeros implies that, for any radical differential ideal $\mathcal{C}\subset \A_1$ and any $g\in\A_1\setminus\mathcal{C}$, there exists a tuple of germs of functions $\mu=\{\mu_i\}_{i=1,\dots,m}$,  such that $p(\mu)\equiv0$, for any $p\in \mathcal{C}$, but $g(\mu)\not\equiv0$.

Let $n=n_{H,\gamma}(1)$ be the universal Noetherianity index, and let
 $m_n(\sigma_{n,n'})\in\A_1$ be a differential polynomial that doesn't belong to the radical differential ideal $\sqrt{\CC_{1,n-1}}$.  Let $\mu$ be a tuple of germs of functions at some point $t_0$ such that $p(\mu)=0$, for any $p\in\sqrt{\CC_{1,n-1}}$,  but  $m_n(\sigma_{n,n'})(\mu)\neq0$.

 We claim that the tuple $\mu$ can be realized as periods of some $\eta\in F_*\Omega^1(U)$, for some sufficiently small $U\subset\C$,  
 \begin{equation}\label{eq:local sharpness}
      \forall t\in U\quad\mu_i(t)=\int_{\delta_i(t)}\eta.
 \end{equation}
Indeed, assume that $\mu$ is defined in some  $\tilde{U}$, with $ t_0\in\tilde{U}$. Choosing, if necessary, another point in  $\tilde{U}$, we can assume that $t_0$ is a regular value  of $F$. Using the algebraic de Rham theorem, choose algebraic forms $\omega_1,\dots,\omega_m$  whose restrictions to $\{F=t_0\}$ form a basis in $H^1(\{F=t_0\},\C)$. Denote by  $P(t)=\{\int_{\delta_i(t)}\omega_j\}_{i,j=1}^m$   the corresponding matrix of periods defined by continuous deformation of $\delta_i(t_0)$. By the definition, $\det P(t_0)\neq0$. Let $U\subset \tilde{U}$ be sufficiently small such that $\det P(t)\neq0$ for all $t\in U$, and let $\lambda(t)=P^{-1}(t)\mu(t)$ be a tuple of functions holomorphic in $u$. Then the form $\eta=\sum\lambda_i\omega_i\in F_*\Omega^1(U)$  satisfies \eqref{eq:local sharpness}.

 Clearly, for  any $p\in\sqrt{\CC_{1,n-1}}$   we have $\ev_\eta(p)=p(\mu)=0$, but $\ev_\eta(m_n(\sigma_{n,n'}))=m_n(\sigma_{n,n'})(\mu)\neq0$. Thus, $\nu^U_{H,\gamma}(1)\ge n=n_{H,\gamma}(1)$, as required.

\end{proof}

\begin{remark}
    Conjecturally, the bound is sharp for Noetherianity index $\nu_{H,\gamma}(1)$ as well, i.e. for the class of polynomial one-forms. This sharpness should necessarily use the invariance of $\CC$ under monodromy, which wasn't used in the above proof.

    \end{remark}
\section{Prospective}\label{Sec:Prospective}

\begin{problem}\hfill
\begin{itemize}
    \item[(i)] Does the  Noetherianity index $\nu_{\scriptscriptstyle H,\gamma}$ coincide with the universal Noetherianity index $n_{\scriptscriptstyle H,\gamma}$? Do they coincide at least in the Examples \ref{exm.Triangle}, \ref{exm:product-lines} and \ref{exm:square}, which we studied? 
    \item[(ii)] In all the examples in which we calculated the  universal Noetherianity index, it was of the form $n_{\scriptscriptstyle H,\gamma}(k)=k+s$, with $s\in\N$. A natural question is, if it is always the case.
\item[(ii)] 
    What is the biggest possible value of $\nu_{\scriptscriptstyle H,\gamma}(k)$, in function of $k$? 
     \end{itemize}
\end{problem}
In this work we studied the relationship between the length of Melnikov functions $M_j$ as iterated integrals and the identical vanishing of the initial Melnikov functions $M_i$, $i=1,\ldots,k$. 
We think that this relationship is important for instance in the study of the following problem. 

\begin{problem}
   Let $\C_{n+1}[x,y]$ be the space of polynomials of degree $n+1$ and $\Omega^1_n$ the space of polynomial 1-forms of degree at most $n$
   
   Given a polynomial $H\in\C_{n+1}[x,y]$,  and a loop $\gamma(t)\in \pi_1( H^{-1}(t),p_0)$, 
   consider the space of deformations \eqref{eq:deformation}, for $\eta\in\Omega_n^1$ and $P^\eta$ the Poincaré map of \eqref{eq:deformation}, along $\gamma$.
   Let $\mu(\eta)$ be the order of $P^\eta-id$, with respect to $\epsilon$. If $P^\eta\equiv id$, we put 
   $\mu(\eta)=\infty$. We define the \emph{weak center order} $\mu=\mu(F, \gamma)$ by 
   $$\mu(H,\gamma)=\max_{\eta\in\Omega^1_n}\{\mu(\eta)<\infty\},\quad \mu(H)=\max_{\gamma}\{\mu(H,\gamma)\}.$$
   Determine $\mu=\mu(H,\gamma)$.

   This problem is analogous to the problem of the maximal weak focus order of polynomial systems of a given order.

   We think that our result is the first step in the solution of the above problem, as we show that the vanishing of the initial Melnikov functions reduces the length of the first non-vanishing Melnikov function, hence simplifying its study. 
\end{problem}

\begin{problem}In \cite{BenNov} an upper bound on the number of zeros of the first non-vanishing Melnikov function $M_k$ of a perturbation of a Hamiltonian system was given in terms of the degrees of the Hamiltonian and the perturbations, and the order $k$ of $M_k$. This bound was a consequence of an upper bound of \cite{BNY} applied to the differential system satisfied by the iterated integrals of length $k$. Reduction of the length (and applying the better bound of \cite{BenNov}) could potentially improve this bound. 
    
\end{problem}
\begin{problem} This article deals with perturbations $dH+\epsilon\eta=0$ of a Hamiltonian system, which means that the return maps are perturbations of the identity mappings. It seems interesting to apply the same reasoning to 
perturbations of Darboux system near a leaf with a non-trivial fundamental group.  In that case, the holonomies of the unperturbed system along the non-trivial loops  are linear maps \cite{Ca}, and their commutators therefore correspond to perturbations of the identity maps as above, see Remark~\ref{rem:perturbations}. The first Melnikov function and its zeros, for deformations of Darboux systems, has been studied in \cite{BM}, \cite{N} etc. under the name \emph{pseudo-abelian integral}.  
\end{problem}


\begin{thebibliography}{CCD13}


\bibitem[A]{Acampo} N. A\textquotesingle Campo. \emph{Le groupe de monodromie du deploiement des singularités isolées de courbes planes
I}, Math. Ann. 213, 1-32 (1975).
\bibitem[BM]{BM} Bobie\'nski, M; Marde\v si\'c, P: \emph{Pseudo-Abelian integrals along Darboux cycles}
Proc. Lond. Math. Soc. (3) 97 (2008), no. 3, 669–688.
\bibitem[BN]{BenNov} Benditkis, S., D. Novikov, D. \emph{On the number of zeros of Melnikov functions}, Ann. Fac. Sci. Toulouse Math. (6) {\bf 20} (2011), no.~3, 465--491.
\bibitem[BD]{BD} Binyamini, G. and Dor, G. \emph{An explicit linear estimate for the number of zeros of Abelian integrals}, Nonlinearity {\bf 25} (2012), no.~6, 1931--1946.
\bibitem[BNY]{BNY} Binyamini, G., Novikov, D. and Yakovenko, S. \emph{On the number of zeros of Abelian integrals}, Invent. Math. {\bf 181} (2010), no.~2, 227--289.
\bibitem[Ca]{Ca}Casale, G. \emph{Feuilletages singuliers de codimension un, groupoïde de Galois et intégrales premières}, Annales de l'Institut Fourier, Volume 56 (2006) no. 3, pp. 735-779.
\bibitem[C]{C} Chen, K.T. \emph{Iterated integrals, fundamental groups and covering spaces}, TAMS, Vol 206, (1975) 83-98.

\bibitem[F]{F} Fran\c coise, J.-P., \emph{Successive derivatives of a first return map, application to the study of quadratic
vector fields}, Ergod. Theory Dyn. Syst. 16 (1) (1996) 87–96.
\bibitem[FP]{FP} Fran\c coise J.-P., Pelletier M.,  \emph{Iterated Integrals, Gelfand—Leray Residue, and First Return Mapping}. J Dyn Control Syst 12, 357–369 (2006).  

\bibitem[G05]{G05} Gavrilov, L. Higher order Poincar\'e-Pontryagin functions and iterated path integrals, Annales de la Facult\'e des Sciences de Toulouse, Vol XIV, No 4, (2005), pp. 663-682
\bibitem[Ily69]{Ily69}  Il'ja\v senko, Ju.S.
\emph{The appearance of limit cycles under a perturbation of the equation $\frac{dw}{dz}=-\frac{R_z}{R_w}$, where $R(z,w)$ is a polynomial}, Mat. Sb. (N.S.), 78(129)(1969), 360--373. 
\bibitem[I69]{I69} Yu. S. Ilyashenko. \emph{An Example of Equations $\frac{dw}{dz}=\frac{P_n(z,w)}{Q_n(z,w)}$
Having a Countable
Number of Limit Cycles and Arbitrarily Large Petrovskii-Landis Genus}. In: Mathematics of the USSR-Sbornik 9.3 (1969), p. 365. doi: 10.1070/SM1969v009n03ABEH001288.
url: http://stacks.iop.org/0025-5734/9/i=3/a=A06.
\bibitem[K57]{K57} Kaplansky, I., \emph{An introduction to differential algebra}, Publications de l’Institut de Mathématique de l’Université de Nancago, Ed. Herman, Paris, 1957.
\bibitem[KOU]{Ki}  Kirillov, A. A.,  Ovsienko V. Yu.,  Udalova O. D., \emph{Identities in the Lie Algebra of Vector Fields
on the Real Line}, Selecta Mathematica Sovietica, Vol. 10, No. 1 (1991).
\bibitem[MNOP]{MNOP} Mardešić, P., Novikov, D., Ortiz-Bobadilla, L., Pontigo-Herrera, J., \emph{Bounding the length of iterated integrals of the first nonzero Melnikov function}, Moscow Mathematical Journal, Vol. 18, Number 2, pp. 367-386, 2018. ISSN: 16093321/ elec.16094514.
\bibitem[MNOP2]{MNOP2} 
Mardešić, P., Novikov, D., Ortiz-Bobadilla, L., Pontigo-Herrera, J., \emph{Infinite orbit depth and length of Melnikov functions}, Ann. Inst. H. Poincaré C Anal. Non Linéaire 36 (2019), no. 7, 1941–1957.
\bibitem[M22]{M22}Movasati, H., \emph{Introduction to Algebraic Curves and Foliations}, author's personal page, IMPA, Brazil, 2022.
 \bibitem[N]{N} Novikov, D. \emph{On limit cycles appearing by polynomial perturbation of Darbouxian integrable systems}
Geom. Funct. Anal. 18 (2009), no. 5, 1750–1773.

\bibitem[P]{PLect} G. Pogudin, \emph{lectures ''Differential equations from the algebraic standpoint''}, mini-course at Institut de Mathématiques d'Orsay, September 2024, \href{www.lix.polytechnique.fr/Labo/Gleb.POGUDIN/files/orsay_course.pdf}{$www.lix.polytechnique.fr/Labo/Gleb.POGUDIN/files/orsay\_course.pdf$}

 \bibitem[PR]{PR} G. Pogudin, Yu. P. Razmyslov, \emph{Prime Lie algebras satisfy the standard Lie identity of degree 5}, Journal of Algebra 468 (2016) 182-192. 
 \bibitem[P90]{P90} H. Poincaré, 
\emph{Sur le problème des trois corps et les équations de la dynamique}, Acta Mathematica
VOL. 13 · NO. 1-2 | 1890
\bibitem[Po]{Po} Pontryagin, L. \emph{On dynamical systems close to hamiltonian ones}, Zh. Exp. \& Theor. Phys. 4 (1934), no.8, 234--238

\bibitem[R34]{[R]} Raudenbush, H. W., Jr. \emph{Ideal theory and algebraic differential equations}. Trans. Amer. Math. Soc. 36 (1934), no. 2, 361–368.
\bibitem[R86]{R86}Yu. P. Razmyslov, \emph{Simple Lie algebras satisfying the standard Lie identity of degree 5}, Math. USSR-Izv., 26:3 (1986), 553–590.
\bibitem[S]{SerreBook} J.-P. Serre, \emph{Lie Algebras and Lie Groups: 1964 Lectures given at Harvard University}, Springer Berlin, Heidelberg, (1992), Series ISSN 0075-8434. 
\bibitem[U1]{U1} M. Uribe, Triangle, \emph{Principal Poincaré–Pontryagin Function of Polynomial Perturbations of the Hamiltonian Triangle}. J Dyn Control Syst 12, 109–134 (2006). https://doi.org/10.1007/s10450-006-9687-4
\bibitem[U]{MU} M. Uribe, \emph{Principal Poincaré-Pontryagin function associated to polynomial perturbations of a product of  $(d+1)$  straight lines}
J. Differential Equations 246 (2009), no. 4, 1313–1341.

\end{thebibliography}
\end{document}